\documentclass[11pt]{amsart}
\usepackage{epsfig}
\usepackage{graphicx, subfigure}
\usepackage{amsmath, amssymb,mathabx}
\usepackage{color}
\usepackage[toc,page]{appendix}
\usepackage{comment}
\usepackage{hyperref}

\newtheorem{thm}{Theorem}[section]

\newtheorem{lem}[thm]{Lemma}
\newtheorem{prop}[thm]{Proposition}

\theoremstyle{definition}

\theoremstyle{remark}
\newtheorem{rem}[thm]{Remark}
\numberwithin{equation}{section}

\newcommand{\eps}{\varepsilon}

\newcommand{\D}{\mathcal{D}}

\newcommand{\OO}{\mathcal{O}}

\newcommand{\U}{\mathcal{U}}

\newcommand{\X}{\mathcal{X}}
\newcommand{\Y}{\mathcal{Y}}
\newcommand{\dist}{\textrm{dist}}

\newcommand{\st}{{\rm s}}           
\newcommand{\un}{{\rm u}}           
\newcommand{\vphi}{\varphi}
\newcommand{\tM}{ \widetilde{\mathcal M}}
\newcommand{\Mv}{ {\mathcal M}^v }
\def\torus{{\mathbb T}}
\def\real{{\mathbb R}}

\def\tLambda{ {\tilde \Lambda}}

\begin{document}
\title{Global Melnikov Theory
in Hamiltonian Systems with General Time-dependent Perturbations}

\author{Marian\ Gidea$^\dag$}
\address{Yeshiva University, Department of Mathematical Sciences, New York, NY 10016, USA }
\email{Marian.Gidea@yu.edu}
\thanks{$^\dag$ Research of M.G. was partially supported by NSF grant  DMS-1515851, and the Alfred P. Sloan Foundation grant G-2016-7320.}
\author{Rafael de la Llave${^\ddag}$}
\address{School of Mathematics, Georgia Institute of Technology, Atlanta, GA 30332, USA}
\email{rafael.delallave@math.gatech.edu}
\thanks{$^\ddag$ Research of R.L. was partially supported by NSF grant DMS-1500943.}

\date{}

\begin{abstract}
We consider a mechanical system consisting of $n$ penduli and a
$d$-dimensional generalized rotator subject to a time-dependent
perturbation.
The perturbation is not assumed to be
either Hamiltonian, or periodic or quasi-periodic,
 we allow for rather  general time-dependence.

The strength of the perturbation is given
by a parameter $\eps\in\mathbb{R}$.
For all $|\eps|$ sufficiently small, the  augmented flow
-- obtained by making the time into a new variable --
has a $(2d + 1)$-dimensional normally hyperbolic locally invariant manifold
$\tilde\Lambda_\eps$.

We define a Melnikov-type  vector,
which gives the first
order expansion of the displacement of the stable and unstable manifolds
of $\tilde\Lambda_0$ under the perturbation.  We provide an explicit formula
for the Melnikov vector in terms of convergent improper integrals
of the
perturbation along homoclinic orbits of the unperturbed system.

We show that if the perturbation satisfies some explicit
non-degeneracy conditions, then the stable and
unstable manifolds of $\tilde\Lambda_\eps$, $W^s(\tilde\Lambda_\eps)$ and
$W^u(\tilde\Lambda_\eps)$, respectively, intersect along a transverse
homoclinic manifold, and, moreover, the splitting of
$W^s(\tilde\Lambda_\eps)$ and $W^u(\tilde\Lambda_\eps)$ can be explicitly
computed, up to the first order, in terms of the Melnikov-type
vector.  This implies that the excursions along some homoclinic
trajectories yield a non-trivial increase of order $O(\eps)$ in the
action variables of the rotator, for all sufficiently small perturbations.

The formulas that we obtain are independent
of the unperturbed motions and give, at the same time,
the effects on periodic, quasi-periodic, or general-type orbits.

When the
 perturbation is Hamiltonian, we express the effects of the
perturbation, up to the first order, in terms of a Melnikov
potential. In addition, if the perturbation is periodic, we obtain
that the non-degeneracy conditions on the Melnikov potential are
generic.

\end{abstract}

\maketitle \markboth{Global Melnikov Theory} {M. Gidea, R. de la Llave}

\section{Introduction}

The goal of this paper is to develop a global Melnikov
theory for the splitting of homoclinic manifolds in Hamiltonian
systems subject to time-\break dependent perturbations. The dependence on
time of the perturbation  includes
\emph{periodic, quasi-periodic, and general time-dependence}.

We adopt a global point of  view and we think of a normally hyperbolic
invariant manifold (NHIM) rather than thinking of
single orbits, hence we obtain  a  theory for all unperturbed orbits irrespective
of whether they are periodic or quasi-periodic or any
other   general behavior.  We know, by
the general theory of
Normally hyperbolic manifolds,
that the NHIM, together with its stable and unstable manifolds,  persist from the unperturbed system  to the perturbed one and that they
depend smoothly on parameters. Hence, our  only focus is
to obtain explicit formulas
 for the splitting of the stable and unstable manifolds, up to the first order,
knowing that these expansions exists.

These formulas can be obtained using only that some of the variables are
slow. In particular, the theory can accommodate \emph{general perturbations,
not necessarily Hamiltonian, including perturbations that exhibit small dissipation}.

Thus, we  develop a global theory that applies to all the orbits,
which is formulated in terms of explicit integrals. The integrals are improper,
but \emph{the integrands converge exponentially fast}, together with several of their
derivatives.

In the special case when the perturbations are Hamiltonian,
using the symplectic geometry of the problem, the obtained formulas can be
expressed very concisely.

Rather than developing a completely general theory, in  this paper
we just consider a very simple model, which generalizes several
models proposed in the literature.   We   develop in full
detail the perturbative theory of homoclinic intersections, and we also relate our results to
the calculations based on the so-called scattering map, which
gives detailed information on the homoclinic excursions.
See
\cite{DelshamsLS00, DelshamsLS08a}. We note that, taking advantage
of the symplectic properties, the scattering map allows to
compute the effect of homoclinic excursions on
all the variables, both fast and slow variables. Other methods
allow only to study the effect on the slow variables.
The theory presented here  relies only
on the difference of scales, and it could work also in other
contexts -- with suitable assumptions -- in other
systems such as parabolic manifolds.

Even if we do not develop it in this paper, we note that
the approach presented here works also in some infinite dimensional
problems or for other types of manifolds
such as singular or parabolic.
 We  stress that we do not need to assume that the
perturbation is periodic or quasi-periodic. One hypothesis that
suffices -- but which could be weakened -- is that the
perturbing  vector-field and a few of its derivatives are bounded.
In future work, we  also hope to apply this approach to random perturbations.

The strategy  of this paper is in sharp contrast to
other presentations of the Melnikov theory
 in the literature,  which try  to combine the derivation of
the formulas with the proof of persistence
of the orbits, which may be cumbersome (sometimes, the presentations assume
that the perturbation vanishes identically on the orbits considered).

Many papers
have to present different treatments
depending  both  on the nature of the unperturbed  orbits (one theory for
periodic orbits, another theory for quasi-periodic orbits)
and on the nature of the time dependence
of the perturbation (there  is one   theory for periodic, and another
one for quasi-periodic perturbations, etc.)
 The method here does not need to distinguish between
the different types of unperturbed orbits and can deal
with
perturbations of different time dependence.  We can see
that many of the problems of the standard Melnikov theory for
quasi-periodic solutions arise from the fact that the orbits themselves
may not persist so that, if one focuses only on quasi-periodic
solutions, it is not clear what the stable/unstable manifolds are
attached to. In our case, the stable/unstable manifolds are
attached to different points of the NHIM. This allows us to
obtain a regular theory. Moreover, it seems to be what is needed
in many applications, e.g to Arnol'd diffusion \cite{GideaLlaveSeara14}.

The reader should be warned that many standard treatments of Melnikov theory,
for example \cite[p. 454 (4.25)]{Wiggins90}, \cite{BeigieLW91, BeigieW92},
give  very different formulas, such as in terms of conditionally convergent
integrals that are supposed to be interpreted as
taking the limit of integration to infinity \emph{along  subsequences of
times}. Unfortunately, finding zeros of the resulting
function is difficult  to interpret physically,
since the value of the integral
would depend on the sequence of times chosen.

\subsection{Model considered in this  paper}
In this section we describe a mechanical system consisting of $n$ penduli and a
$d$-dimensional generalized rotator, which are weakly coupled via a small, time-dependent perturbation of a general type.

\subsubsection{Unperturbed system.}
The phase space for the unperturbed system is the symplectic  manifold $M=\mathbb{R}^n\times\mathbb{T}^n\times \mathbb{R}^d\times\mathbb{T}^d$,  endowed with the standard symplectic form
\[\Upsilon=\sum_{i=1}^{n}dp_i\wedge dq_1+\sum_{j=1}^{d}dI_j\wedge d\vphi_j, \] where each point $x\in M$ is uniquely described by its coordinates $x=x(p,q,I,\vphi)\in \mathbb{R}^n\times\mathbb{T}^n\times \mathbb{R}^d\times\mathbb{T}^d$. Here $\mathbb{T}=\mathbb{R}/\mathbb{Z}$.
We denote
$p=(p_1,\ldots, p_n)\in \mathbb{R}^{n}$, $q=(q_1, \ldots, q_n)\in \mathbb{T}^{n}$, $I=(I_1,\ldots,I_d)\in\mathbb{R}^d$, $\vphi=(\vphi_1,\ldots,\vphi_d)\in \mathbb{T}^d$.

On $M$ we consider the autonomous Hamiltonian:
\begin{equation}\label{eqn:hamiltonian}
\begin{split}
H_0(p,q,I,\vphi,t)=h_0(I)+\sum_{i=1} ^{n}\pm\left (\frac{1}{2}p_i^2+V_i(q_i)\right).
\end{split}
\end{equation}
In the above, we assume that  each potential $V_i$ is periodic of period $1$ in $q_i$.

We denote the corresponding Hamiltonian vector field  by
\begin{equation}\label{eqn:H_0VF}
\X^0=J\nabla_x H_0,
\end{equation}
where \[J=\left(
            \begin{array}{cc}
              0 & -I \\
              I & 0 \\
            \end{array}
          \right)
.\]

The corresponding Hamilton equations are given by
\begin{equation}\label{eqn:unp}
\dot x=\X^0(x),
\end{equation}
and the corresponding Hamiltonian flow on $M$ is denoted by $\phi^t_\eps$.

The term $h_0(I)$ represents an integrable Hamiltonian system in the
variables $(I,\vphi)$, and can be viewed as a generalized $d$-dimensional
rotator. The action coordinate $I$ is preserved by the Hamiltonian
flow of $h_0$, so each solution of the flow of $h_0$ is confined to a
$d$-dimensional torus. The motion on each torus is given by a linear
flow of frequency vector $\omega(I)=({\partial h_0}/{\partial I})(I)$.
Hence the dynamics in each coordinate $\vphi_j$ is a rotation of
frequency $\omega_j(I)$, $j=1,\ldots,d$.

The terms $P_i(p_i,q_i)=\pm\left (\frac{1}{2}p_i^2+V_i(q_i)\right)$, $i=1,\ldots,n$, are  Hamiltonians in the variables $(p_i,q_i)$, and  they describe a system of $n$ penduli. Under some general conditions on the potentials $V_i$ (see assumption \textbf{H2} below), each
pendulum $P_i$ has a hyperbolic fixed point whose stable and unstable manifolds coincide, forming a separatrix in the phase $\{(p_i,q_i)\}$ of $P_i$.
Sometimes the penduli are called
integrable, but note that the action variables are singular
on the separatrices, so that they are not integrable in the sense
of admitting  smooth action angle variables.

A caricature of a
mechanical system as in \eqref{generalperturbation}, showing one pendulum
and one rotator,   in which the coupling is imagined as a varying
magnetic field, is illustrated in Figure \ref{pendulumrotor}.

\begin{figure}
\includegraphics[width=0.4\textwidth]{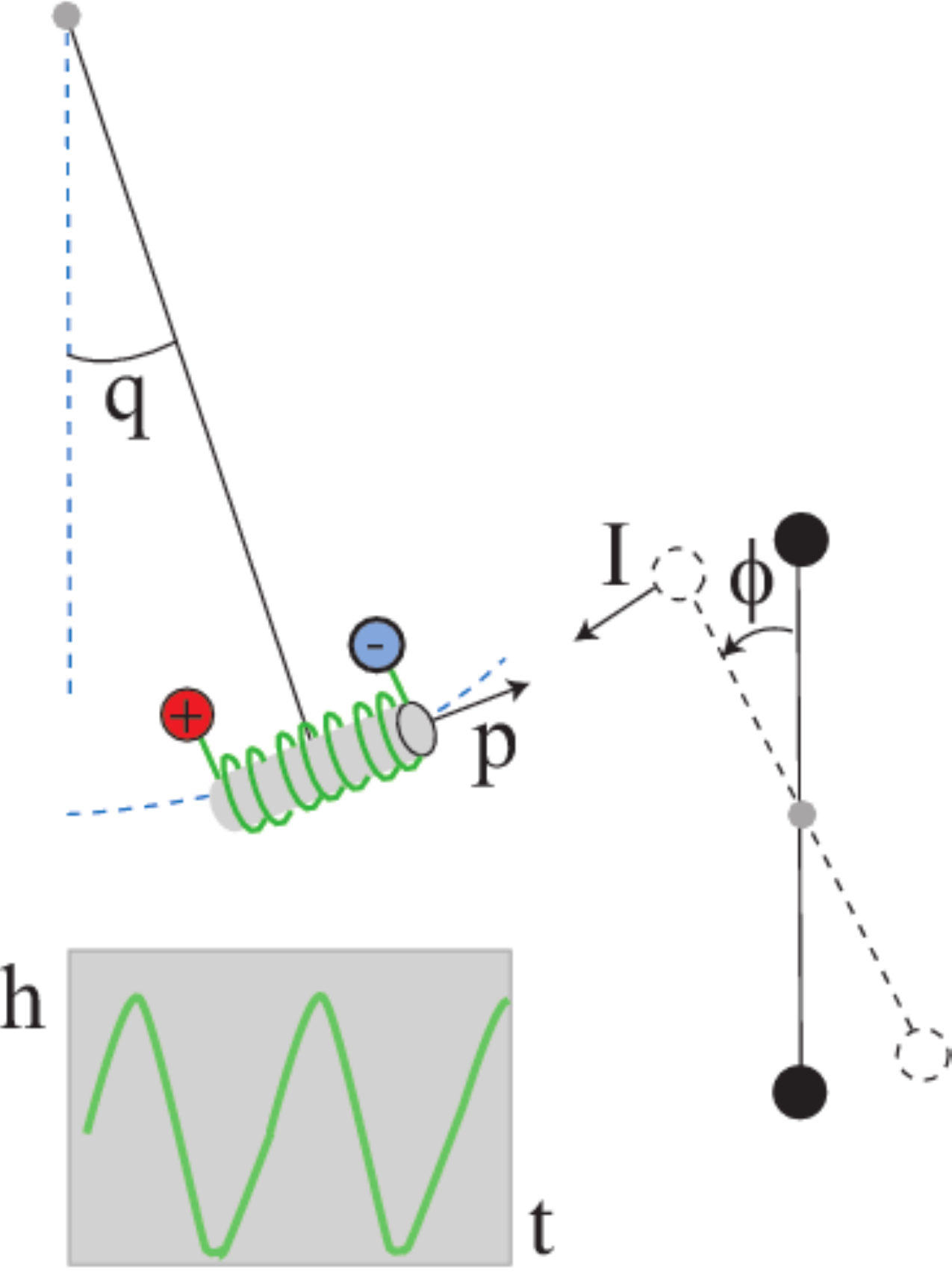}
\caption{Weakly coupled pendulum and rotator.}
\label{pendulumrotor}
\end{figure}

The model \eqref{eqn:hamiltonian}
appears as  a normal form near a simple resonance of a typical Hamiltonian system. Models of this form have been considered in many works, among them \cite{HolmesM82,Wiggins90,DLS_Multi}.

Note that we allow   the sign  in front of the
total energy function  to be indefinite.
Models with Hamiltonians which are neither positive
definite nor negative definite appear naturally in the theory of interaction of waves,
such as  in  ``Krein cruch'' models, see \cite{MacKay}
and in many other models.
For our treatment, the signs of the penduli do not make
any difference, but it seems that they make a crucial difference
for variational methods used in Arnol'd diffusion.

\subsubsection{General perturbation.} We consider a general time-dependent perturbation of the system \eqref{eqn:H_0VF}, given by
\begin{equation} \label{generalperturbation}
\dot x = \X^0(x) + \eps \X^1(x,t;\eps),
\end{equation}
where $\X^1$  is a time-dependent, parameter-dependent vector field on $M$, and  $\eps\in\mathbb{R}$ is a parameter representing the size of the perturbation.
On $\X^1$  we require only  some mild regularity
and boundedness assumptions that we will specify later.
In particular,  we allow that the perturbation is not Hamiltonian -- e.g., when
the mechanical system has some small friction.

As usual in differential geometry, we will consider vector fields also as
differential operators; that is,  given a smooth vector field $\X$ and a smooth function $f$ on the
phase space $M$,
\begin{equation}
\label{eqn:vf_derivative}(\X  f)(x) = \sum_j (\X)_j (x) (\partial_{x_j} f)(x),
\end{equation}
where the subindex $j$ denotes the components.

We denote by $\phi^{t_1, t_0}_\eps$ the non-autonomous
flow of \eqref{generalperturbation}.
That is, $\phi^{t_1, t_0}_\eps$ is the map that, given initial condition
at $t = t_0$, produces the solution of \eqref{generalperturbation} evaluated at
at $t = t_1$. Note that
$\phi^{t_2, t_1}_\eps \circ \phi^{t_1, t_0}_\eps
= \phi^{t_2, t_0}_\eps $.

In is convenient to analyze the system in the augmented phase space
\[\tilde{M}=\mathbb{R}^{n} \times \mathbb{T}^{n}\times\mathbb{R}^d  \times \mathbb{T}^d\times\mathbb{R}.\]
A point in $\tilde M$ will be denoted by $\tilde{x}=\tilde{x}(p,q,I,\vphi,t)=\tilde{x}(x,t)$. The  non-autonomous flow of \eqref{generalperturbation} gives rise to an augmented flow $\tilde{\phi}^\sigma_\eps $ on $\tilde{M}$ given by
\begin{equation}\label{eqn:extended_hamiltonian}
\tilde{\phi}^\sigma_\eps(x,t)=(\phi^{t + \sigma, t}_\eps(x),t+\sigma), \textrm{ for } \sigma \in\mathbb{R}.
\end{equation}

We will assume  that the solutions of the equations \eqref{generalperturbation}   are defined in
a domain $\D \subset \tilde{M}$ given by
\begin{equation} \label{domain}
\D = \{ (p, q, I, \vphi, t) \, | (p,q) \in \OO_1, I \in \OO_2, \vphi \in \torus^d, t \in \real \},
\end{equation}
and we assume $C^k$-regularity
and uniform bounds for the perturbation $\X^1$ and its derivatives in the domain $\D$.
Here $\OO_1\subset \real^d \times \torus^d$ is a neighborhood of the family of separatrices of the penduli
(discussed in more detail in \eqref{separatrix}), and $\OO_2\subseteq\real^d$ is an open
ball.  We will assume that the boundaries of $\OO_1, \OO_2$
are smooth (this will simplify the extension arguments).
Of course, in many applications the
system will be defined in all $\tilde M$.

\subsubsection{Hamiltonian perturbation.} We will
give special attention to
the particular case when the perturbation in \eqref{generalperturbation} is Hamiltonian, that is, is given by

\begin{equation}\label{eqn:h}
   \X^1(x,t;\eps)=J\nabla_x h(x,t;\eps),
\end{equation}
where $h$ is a time-dependent, parameter-dependent Hamiltonian function.
In this case, the solutions of \eqref{generalperturbation} are the solutions of the Hamilton equations for the Hamiltonian \begin{equation}\label{eqn:hamiltonian_eps}
\begin{split}
H_\eps(p,q,I,\vphi,t)=h_0(I)+\sum
_{i=1} ^{n}\pm\left (\frac{1}{2}p_i^2+V_i(q_i)\right)+\eps
h(p, q, I,\vphi, t; \eps).
\end{split}
\end{equation}

\subsubsection{Overview on the geometric effect of perturbations.}
As mentioned before, the unperturbed system, given by the Hamiltonian \[H_0(p,q,I,\vphi)=h_0(I)+P(p,q):=h_0(I)+\sum_{i=1}^{n}\pm P_i(p_i,q_i),\] can be viewed as the product of $d$ rotators and $n$ penduli.
Under some general conditions on the potentials $V_i$ (see assumption \textbf{H2} below), each pendulum $P_i$ has a hyperbolic fixed point whose stable and unstable manifolds coincide.
Fixing the pendulum variables
to be at a hyperbolic fixed point, and restricting the action variable
to some closed ball in the action space, we obtain a $(2d)$-dimensional annulus \begin{equation}\label{Lambda0}
\Lambda_0 = \{ (I, \vphi) \, | I \in B \subset \real^d,\,  \vphi \in \torus^d \}
\end{equation}
 which is foliated by $d$-dimensional tori invariant under the flow of the unperturbed system.
These invariant tori represent level sets in $\Lambda_0$ of the action $I$.
Each of these tori has homoclinic excursions to itself, along the separatrices of the penduli.

In the augmented phase-space, there is a corresponding $(2d+1)$-dimensional normally hyperbolic invariant manifold \begin{equation}\label{tildeLambda0}
\tilde\Lambda_0 = \{ (I, \vphi,t) \, | I \in B \subset \real^d,\, \vphi \in \torus^d ,\, t\in\real\}.
\end{equation}

There are no orbits in the unperturbed system that
originate near one level set in $\Lambda_0$ of the action variable   and end  up near a different  level set of the action variable. This is not the case, in general, when the perturbation is turned on.

In Section~\ref{sec:reminder} we will show  that
the annulus $\tilde\Lambda_0$ described above
is a normally hyperbolic manifold in the sense of \cite{Fenichel71, HirschPS77, Pesin04}.
 The theory of normally
hyperbolic invariant manifolds shows that, for every small enough
$\eps$  there exits a  manifold $\tilde\Lambda_\eps$
diffeomorphic to the original invariant manifold
which is invariant under the flow of $H_\eps$.
The perturbed manifold $\tilde\Lambda_\eps$ possesses
$(2d+n+1)$-dimensional stable and unstable manifolds.

The main contribution of this paper is to develop a Melnikov
vector, which measures the effect of the perturbation  on the
homoclinic orbits of the unperturbed system, and to show the
existence, under certain conditions, of trajectories of the perturbed
system that change the action variable by some non-trivial amount, $\OO(\eps)$
When the perturbation is Hamiltonian, we show
that the Melnikov vector derives from a Melnikov potential.

More precisely, we show that if the Melnikov vector (or potential) satisfies
certain non-degeneracy conditions, which are given explicitly and are
satisfied by generic perturbations, then there exit homoclinic orbits
that start near one level set of the action variable in  $\tilde\Lambda_\eps$  and end up close
to another level set of the action variable  $\tilde\Lambda_\eps$, at a distance of order $O(\eps)$ apart.

The methods used in this paper are standard (and
rather elementary)  methods of
perturbation theory for systems with slow variables.
We note that this part of the argument allows
also that the perturbation is not Hamiltonian.
This is a problem that appears in several practical problems
with small dissipation.

In the case that the perturbation is Hamiltonian, we can
use the geometry to express the results more concisely.
This formalisms has
qualitative consequences.  Using the symplectic
 geometry, there is a relation between critical points of
functions and intersections. Then, using
Morse theory, one can conclude existence of intersections.
Of course, for more general perturbations considered here
(e.g. dissipative), the Melnikov vector could fail
to come from a potential.

We note that there is an alternative and more sophisticated
point of view, the theory of the scattering maps to normally
hyperbolic manifolds \cite{DelshamsLS08a}.  The theory of the scattering map
uses significantly more
sophisticated tools than the perturbation theory for
slow variables, but we will see leads to the same results for
the slow variables. The scattering map theory also produces
results for the fast variables. In this paper, therefore,
we include a comparison. We hope that the use of the
elementary perturbation theory could be accessible to a wider audience.

\subsubsection{Relation with  Arnold diffusion.}
An important question is whether  changes in the action of order
$O(\eps)$, as established in this paper, accumulate to a change of order $O(1)$, i.e, which is
independent of the size of the perturbation.

 The
interpretation of the orbits that drift by a distance of order $O(1)$
in the action variable is that, for arbitrarily
small perturbations, there exist initial conditions on the
mechanical system on which the rotator spins faster and faster over
time. This question is at the heart  of the Arnol'd diffusion,
which is experiencing great progress these days.  We will not discuss
this problem here, but, of course, it is one of the motivations
for the present research. With the goal in mind of Arnol'd diffusion,
obtaining the information on the effect of
homoclinic excursions in the fast variables is
a useful information. See \cite{GideaLlaveSeara14}.

\subsubsection{Organization of the paper.}
In Section~\ref{sec:main_results} we formulate a sequence of
hypotheses of increasing restrictiveness  and state the main results
which use subsets of the hypotheses  stated before.
In Section~\ref{section:theproofs} we present the proofs
of the results.  In Appendix~\ref{sec:hyperbolicity},
we collected some of the main results of the
theory of Normally Hyperbolic Invariant Manifolds
(henceforth, NHIM).
 This appendix is mainly to set up
the notation.
In Appendix \ref{sec:extension} we discuss the extension of a flow or vector field defined on a non-compact manifold with boundary, in order to obtain a  result on the persistence of normally hyperbolic locally invariant manifold.
In Appendix~ \ref{sec:scatteringnhim},
we present an alternative approach for Hamiltonian systems
based on the scattering map.
Finally, in Appendix~\ref{sec:nonautonomous}, we present
another formalism for time-dependent perturbations which
involve recurrence. This formalism  has been used in \cite{GideaL17},
where it is shown
that the recurrence in the perturbation can
lead to accumulation of effects in some models.

\section{Set-up and main results}\label{sec:main_results}

We will describe several assumptions on the system
\eqref{eqn:hamiltonian}, and on the perturbation $\X^1$, listed below. Then we will state the main
results of the paper, under the appropriate subset of the  assumptions
listed below.  We note that these assumptions can be readily verified
in concrete models of interest.

We have chosen to present the assumptions  in a cookbook manner
so that they can be used by the practitioner. The motivation,
physical interpretation will be discussed in Section~\ref{section:theproofs}

Later, in Section~\ref{sec:genericity} we
will show that the assumptions leading to
the existence of transverse intersections are generic (indeed, $C^2$ open
and $C^\infty$ dense) in the case that the perturbations are Hamiltonian
and periodic.  We think that the assumption of Hamiltonian
indeed does belong since adding dissipative perturbations can lead to
manifolds that doe not intersect. On the other hand, we
believe that the assumption of periodicity can be relaxed.

\subsection{Assumptions for subsequent results}
\label{sec:assumptions}

\subsubsection{Assumptions \textbf{H1} and \textbf{H2}.}{$ $}

The hypothesis \textbf{H1} will be divided into two parts:
\textbf{H1.1} and \textbf{H1.2}

\textbf{H1.1} \textit{The terms $V_i$, $h_0$ in \eqref{eqn:hamiltonian}
are $C^{k+1}$-smooth with $k\geq k_0$, for some $k_0\geq 2$ suitably large.
We assume that the derivatives of $V_i$ are uniformly
bounded in $q_i\in\torus$, and that the derivatives of $h_0$ are
bounded in an open set. }

\textbf{H1.2}
\textit{We assume that the perturbation   $\X^1$, in the general case,  (resp.
$h$ in the Hamiltonian case) is defined on the domain $\D$ from \eqref{domain}, and
is $C^k$ bounded (resp. $C^{k+1}$ boundend) there.}

\smallskip

The boundedness assumption in \textbf{H1.2} on the derivatives of $h$
is satisfied automatically if the dependence of $h$ on $t$ is periodic
or quasiperiodic, which are the cases considered in the literature
before.

\smallskip

\textbf{H2} \textit{Each potential $V_i:\torus\to \mathbb{R}$ has a
non-degenerate local  maximum (which, without loss of generality,
we set at $0$). Hence,  $V''_i(0)<0$.}

\smallskip

Hypothesis \textbf{H2} implies
that each pendulum has a homoclinic orbit to $(0,0)$, which we parametrize as $(p^0_i(t),q^0_i(t))$, $t\in\mathbb{R}$. Note that the parameterization is unique up to the choice of origin of time $t$.
Denote $V''_i(0) = -(\lambda_i)^2$, and  $\lambda_+ = \min_i \lambda_i$.  This means
that homoclinic orbits converge (both in the future and in the past) to $(0,0)$ with
a exponential rate which is at least $\lambda_+$.

In the  system \eqref{eqn:hamiltonian}
 there is an $n$-dimensional
family of homoclinic orbits for the whole system of penduli given by
\begin{equation}\label{separatrix}
\begin{split}(p^0_\tau(t),q^0_\tau(t))&=(p^0,q^0)(\tau+t\bar 1)\\&=(p^0_1(\tau_1+t), \ldots, p^0_n(\tau_n+t),
q^0_1(\tau_1+t), \ldots, q^0_n(\tau_n+t)),\end{split}
\end{equation}
 where
$\tau=(\tau_1,\ldots,\tau_n)\in\mathbb{R}^n$  and  $\bar 1=(1,\ldots,1)\in\mathbb{R}^n$.
The parameters
$\tau_1, \ldots, \tau_n$ represent the initial times (phases) of the
motion of the individual penduli.

We can think of the homoclinic manifold as a product of
the homoclinic orbits of each
 penduli. The parameters $\tau_i$ indicate the
time at which the jump of the pendulum $i$ takes place (for example,
identify the time where it reaches a maximum). In a non-autonomous
perturbation, the energy gained in an excursion may depend on
which time each of the penduli jumps.

\subsubsection{Melnikov vector and potential.}
Using the homoclinic intersections and the perturbations, we
will define the Melnikov integrals. They are the crucial
objects for this paper. Some of the subsequent hypothesis
will involve their properties.

In the case of general perturbations \eqref{generalperturbation},
we define
the $n$-dimensional  \emph{Melnikov vector} $\Mv=(\Mv_i)_{i=1,\ldots,n}$, of components
\begin{equation} \label{eqn:melnikov_vect}
\begin{split}
\Mv_i(\tau, I, \vphi, t) =
\displaystyle-\int
_{-\infty}^{\infty}&\left[
(\X^1 P_i)(p^0(\tau+\sigma\bar 1),q^0(\tau +\sigma\bar 1),I,\vphi+\sigma\omega(I),t+\sigma;0)\right .\\
&\left .-(\X^1P_i)(0,0,I,\vphi+\sigma\omega(I),t+\sigma;0)\right
]d\sigma.
\end{split}
\end{equation}

Explicitly, for $\tilde x=\tilde{x}(p,q,I,\phi,t)$, if we write $\X^1$ in terms of its components as \[\X^1=((\X^1)^{p_1},(\X^1)^{q_1},\ldots,(\X^1)^{p_n},(\X^1)^{q_n}, (\X^1)^{I_1},(\X^1)^{\phi_1},\ldots, (\X^1)^{I_d},(\X^1)^{\phi_d}),\] then $\X^1P_i$ is given, due to \eqref{eqn:vf_derivative}, by \[(\X^1P_i)(\tilde x)=(\X^1)^{p_i}(\tilde x)(\partial _{p_i}P_i)(p_i,q_i)+(\X^1)^{q_i}(\tilde x)(\partial _{q_i}P_i)(p_i,q_i).\]

\begin{rem}
Note that the Melnikov vector, of dimension $n$,
has coordinates which are integrating a function in
phases space (namely $X^1P_i$, a directional derivative of
the function $P_1$) over a homoclinic orbit and subtracting the
integral over the asymptotic orbit.

It depends on the phases $\tau_i$ that specify the homoclinic
excursion and the final points of the orbit as well as the  time $t$.

As we will see in Section~\ref{section:theproofs},
the physical meaning of the Melnikov
vector $\Mv$ in \eqref{eqn:melnikov_vect}, is that it is
the expression in some coordinates of
the leading order in $\eps$ of the splitting
the stable and unstable manifolds.   The coordinates will
be specified in Section~\ref{section:theproofs}.

This will make plausible our main result,
Theorem~\ref{prop:melnikov_jump}, which shows
that if this leading order has a non-degenerate zero,
we obtain a transversal intersection for small enough $\eps$.
\end{rem}

\medskip

When the perturbation is Hamiltonian, as in \eqref{eqn:h},
we will see that the Melnikov vector is the symplectic gradient
of a function (as it often happens in a function, all small
effects can be related to gradients of functions).
See Remark~\ref{rem:poisson_integral}.

We define the  \emph{Melnikov potential}  for the homoclinic
family  $(p^0_\tau,q^0_\tau)$  by
\begin{equation}\label{eqn:melnikov_int}
\begin{split}
\tM(\tau,I,\vphi,t)=\displaystyle-\int
_{-\infty}^{\infty}&\left
[h(p^0(\tau+\sigma\bar 1),q^0(\tau +\sigma\bar 1),I,\vphi+\sigma\omega(I),t+\sigma;0)\right .\\
&\left .-h(0,0,I,\vphi+\sigma\omega(I),t+\sigma;0)\right]d\sigma.
\end{split}
\end{equation}

\begin{rem}
Note that, since
$(p^0(\tau + \sigma {\bar 1}),
p^0(\tau + \sigma {\bar 1}) )$ approaches $(0,0)$, as $\sigma \to \pm\infty$,
exponentially fast, the integrand  in \eqref{eqn:melnikov_vect}, as well as
 in \eqref{eqn:melnikov_int},
converges to $0$ exponentially fast and the integral is
absolutely convergent.
A more detailed study
of the convergence of
the integrand and its derivatives
is taken up in  Proposition \ref{prop:melnikov_int} below.
\end{rem}

\begin{rem}\label{rem:poisson_integral}
The fast convergence of the integrands and
their derivatives  allows us to take derivatives out of the
the improper integrals.

In the Hamiltonian case, we have \begin{equation}(\X^1P_i)=(-\partial_{q_i} h) (\partial _{p_i}P_i) +(\partial_{p_i} h) (\partial _{q_i}P_i)  = [P_i, h],\end{equation}
with
$[ \cdot, \cdot ]$ the Poisson bracket.
Using that the time-dependent Hamiltonian flow
preserves the Poison bracket $[P_i, h]\circ \phi^{t,s}  =
[P_i \circ \phi^{t,s}, h\circ \phi^{t,s}]$,
we can take  the Poisson bracket out of the integral
defining the Melnikov vector, and  we have
in the Hamiltonian perturbation case,
\begin{equation}\label{geometrymelnikov}
 \Mv_i = [P_i, \tM] = \partial_{\tau_i} \tM.
\end{equation}
The (elementary) last identity in \eqref{geometrymelnikov} will
be discussed in more detail in \eqref{tauderivative}.
The usefulness of \eqref{geometrymelnikov}
is that it relates the Melnikov vector
to the gradient of a function. When the perturbation is periodic, we
will be able to show that  the Melnikov function has some periods in $\tau$
so that critical points have to exist.
\end{rem}

\subsubsection{Assumptions \textbf{H3} and \textbf{H4}.}
The next assumption
is in terms of the Melnikov vector \eqref{eqn:melnikov_vect}, or
the Melnikov potential \eqref{eqn:melnikov_int}.

\smallskip

\textbf{H3.} \it The Melnikov vector associated to the homoclinic
family $(p^0_\tau,q^0_\tau)$ satisfies the following non-degeneracy
condition.

There exists an open  ball $\mathcal{I}^*\subseteq \OO_2\subseteq\mathbb{R}^d$ and an open set $\mathcal{O}\subseteq\mathbb{T}^d\times\mathbb{R}$
such that for every $(I,\vphi,t)\in \mathcal{I}^*\times\mathcal{O}$
the map
\[\tau\in\mathbb{R}^n \to
\Mv(\tau,I,\vphi,t)\in\mathbb{R}^n \]
has a non-degenerate
zero $\tau^*$ (that is,  the $(n \times n)$ matrix
$D_\tau \Mv( \tau^*)$ has full rank $n$)
 which is locally given, by the
implicit function theorem,  in the form
\[\tau^*=\tau^*(I,\vphi,t),\]
for $(I,\vphi,t)\in \mathcal{I}^*\times\mathcal{O}$.

In the Hamiltonian case, the above assumption
is equivalent, by Remark~\ref{rem:poisson_integral}, to assuming
that
\[\tau\in\mathbb{R}^n \to
\tM(\tau,I,\vphi,t)\in\mathbb{R}\]
has a non-degenerate critical point at $\tau^*$. \rm

\smallskip

Consider now the case when the perturbation is Hamiltonian, as in  \eqref{eqn:h}.

Let \[ \mathcal{M}^*(I,\vphi,t)=\tM(\tau^*(I,\vphi,t),I,\vphi,t).\]

Consider the  mapping
\begin{equation}\label{eqn:m_reduced} {\widetilde{\mathcal{M}}^*}(I,\theta)={\mathcal{M}}^*(I,\theta,0),\end{equation}
defined for all $ (I,\theta)$ with $I\in\mathcal{I}^*$  and $\theta\in\mathbb{R}^d$ of the form $\theta=\vphi-t\omega(I)$, that is, for all $\theta\in\mathbb{R}^d$ for which $(\theta+t\omega(I),t) \in \mathcal{O}$ for some $t$. The function ${\widetilde{\mathcal{M}}^*}$ will be referred to as the  reduced Melnikov potential. The relation between the Melnikov potential and the reduced Melnikov potential will be explained in further detail in Section \ref{section:reduced}.

We will denote by $\partial {\widetilde{\mathcal{M}}^*}/\partial I$
the derivative of ${\widetilde{\mathcal{M}}^*}$ with respect to the
first variable, i.e.,
$\partial {\widetilde{\mathcal{M}}^*}/\partial I =\partial_1 {\widetilde{\mathcal{M}}^*}$, and by $\partial
{\widetilde{\mathcal{M}}^*}/\partial \theta$ the derivative of
${\widetilde{\mathcal{M}}^*}$ with respect to the second variable,
i.e., $\partial {\widetilde{\mathcal{M}}^*}/\partial \theta=\partial_2
{\widetilde{\mathcal{M}}^*}$.

The next assumption  is in terms of the derivative of the reduced Melnikov potential. An advantage of the reduced Melnikov potential in comparison to the un-reduced one is that the former is defined at points in the phase space of the system of rotators $\mathbb{R}^d\times\mathbb{T}^d$, while the latter is defined at points in the augmented phase space.  More precisely,   $\frac{\partial  {\widetilde{\mathcal{M}}^*}}{\partial \theta}$ is defined at the points $(I,\theta)\in \mathbb{R}^d\times\mathbb{T}^d$ where $I\in \mathcal{I}^*$ and $\theta=\vphi-t\omega(I)$ for some $t$.

\smallskip

\textbf{H4.} \it  The  reduced Melnikov potential associated to the homoclinic
orbit $(p^0(\tau),q^0(\tau))$ satisfies the following non-degeneracy
condition.

There exists an open ball $\mathcal{I}^{**}\subseteq \mathcal{I}^{*}$  such that for each  $(I,\vphi, t)\in\mathcal{I}^{**}\times\mathcal{O}$, we have   \begin{equation}\label{eqn:m_trans}
\frac{\partial{\widetilde{\mathcal{M}}^*}}{\partial\theta}(I,\vphi-t\omega(I))\neq 0.
\end{equation}
\rm

\subsection{Main results}

Now we present the main results of the paper.

\subsubsection{Generalities on normally hyperbolic invariant manifolds}
Let $\mathcal{I}\subseteq \mathbb{R}^d$ be an open ball,
and let
\begin{equation}
\label{eqn:lambda_00}
\Lambda_0=\{(I,\vphi)\,|\,I\in \textrm{cl}(\mathcal{I}),\,\vphi\in\mathbb{T}^d\},
\end{equation}
which is a compact manifold with boundary for the flow $\phi^t_0$.
In the augmented phase space, $\tilde\Lambda_0=\Lambda_0\times \real\subseteq \tilde{M}$ is a non-compact, normally hyperbolic invariant manifold with boundary for the  augmented  flow $\tilde\phi^t_0$. By our assumption \textbf{H1},  the vector fields $\X^0$ and $X^1$  are uniformly bounded and uniformly continuous, together with their derivatives.
We note that while the standard  theory of normally hyperbolic
manifolds requires the unperturbed manifold to be compact, in fact compactness only
enters into the arguments through the fact the vector fields are uniformly bounded and uniformly continuous.
Thus,  if one assumes instead of compactness only uniform continuity, the
standard results on the persistence of normally hyperbolic manifolds go through.

The papers \cite{Fenichel71, HirschPS77} established    that
if we perturb the original  $\tilde\Lambda_0$ survives as a \emph{normally hyperbolic locally  invariant manifold}
$\tilde\Lambda_\eps$ for all $\eps$ sufficiently small,  which can be parametrized over $\tilde\Lambda_0$ and which depends
differentiably on $\eps$.  There are, however, some subtleties
that we recall. More modern references are \cite{Pesin04, BatesLZ08}.

The proof of \cite{Fenichel71, Fenichel74} involves extending the vector field in a suitable way
and showing that the extended vector field has a unique invariant
manifold. This invariant manifold for the extension is a locally
invariant manifold for the original problem.

It is important to
remark that the invariant manifold for the extension may depend
on the extension chosen.
Hence, the locally invariant manifold $\tilde\Lambda_\eps$ for the original
problem is not unique.

If one chooses an  extension of the vector
field  which depends smoothly on
parameters, one
obtains a manifold that depends smoothly on parameters.
The exact degree of smoothness, depends on the rates of
growth of derivatives in the manifold \cite{Fenichel74}.
See \eqref{eqn:ratesdifferentiable} in  Appendix~\ref{sec:hyperbolicity}.

In our case, when the motion on the manifold is integrable,
one gets that the invariant manifold and its stable manifolds
has any fixed number of derivatives for
sufficiently small $\eps$ -- the smallness conditions depend
on the order required. So, general invariant manifolds may be
only finitely differentiable. See \cite{Fenichel74} for examples.
 Assuming that the motion on the manifold
is a rotation -- an assumption that may require adjusting parameters --
one can get analytic invariant manifolds \cite{CanadellH17a}.

In the main part of this paper, we will use the notation  $C^k$ regularity
for regularities that can be made as high as desired
by making $\eps$ large enough. We warn the reader that
the value of $k$ may change slightly from line to
line in some of the calculations (e.g. the regularity of
the bundles may be one less than that of the manifold or
the derivatives have one derivative less than the functions, etc.)

We note that one of the consequences of
the result is that, for systems like ours
in which the dynamics in the manifold is integrable,
the perturbation of the manifold is independent of
the extension chosen to all orders in $\eps$.  In more
general cases, the perturbation theory does depend on the
extension.

Similarly,   the Melnikov integrals \eqref{eqn:melnikov_int}
and \eqref{eqn:melnikov_vect} are expressions which only involve
the unperturbed system and they are independent of the extensions
used. They are independent of the family of locally invariant
manifolds considered.  The same happens to all orders in the
perturbation in $\eps$.

For more details on this section, see Appendix \ref{sec:extension}.

\subsubsection{Statement of results.}

\begin{prop}\label{prop:nhim}

Assume that the system \eqref{generalperturbation} satisfies the conditions
\textbf{H1} and \textbf{H2}.  Fix $k \in \mathbb{N}$.  Choose a
$C^k$ family of extended flows on $\tilde{M}$.

Then,  there
exists $\eps_0$ sufficiently small such that for each
$0<|\eps|<\eps_0$, there is a  unique $C^k$ family of manifolds
$\tilde\Lambda_\eps$ invariant and normally hyperbolic
for the extended  flow
$\tilde{\phi}^t_\eps$
(hence, locally invariant for the original flow).
\end{prop}

\begin{rem}\label{parameterization1}
As a consequence of the implicit function theorem,  there exists a
$C^{k}$-family of diffeomorphisms $\tilde \Xi_\eps:\tilde\Lambda_0\to
\tilde{M}$ such that
$\tilde\Lambda_\eps=\tilde\Xi_\eps(\tilde\Lambda_0)$.

Using the implicit function theorem, one can ensure that the
$\Xi_\eps$ has additional properties, such as
$\frac{d}{d\eps} \Xi_\eps|_{\eps = 0} \in \textrm{Span}( E^\st \oplus E^\un)$.
\end{rem}

The following result shows that the Melnikov potential
\eqref{eqn:melnikov_int} is well defined, and that its non-degenerate
critical points correspond to transverse intersections of the stable
and unstable manifolds of $\tilde\Lambda_\eps$ for the extended flow.
We emphasize that one of the conclusion is that, for the models
we are considering, the conclusions are given by the formula
\eqref{eqn:melnikov_int} and, in particular, the Melnikov function
does not depend on  the extensions of the flow used in  the construction
of the locally invariant manifolds.

\begin{prop}\label{prop:melnikov_int}
Assume that the system \eqref{generalperturbation}   satisfies the conditions \textbf{H1},
\textbf{H2}.   Then,  the integral \eqref{eqn:melnikov_vect}
in the general case (resp., $\tM$ in the Hamiltonian case) is convergent
and, moreover, the vector  $\Mv$ (resp., $\tM$) is a $C^k$
function of all of its arguments.

In addition, if  we assume \textbf{H3},
there exists $\eps_0$ sufficiently small such that,
for each $0<|\eps|<\eps_0$, the stable and unstable manifolds  of $\tilde\Lambda_\eps$, $W^s(\tilde\Lambda_\eps)$ and $W^u(\tilde\Lambda_\eps)$, respectively,  intersect transversally,  in the augmented  phase space $\tilde{M}$, along
a homoclinic manifold $\tilde\Gamma_\eps$, which can be parametrized by
\begin{equation}\label{eqn:gamma_eps}\tilde\Gamma_\eps =\{{\tilde x}={\tilde x}(\tau^*(I,\vphi,t)),I,\vphi,t)\,|\,(I,\vphi,t)\in\mathcal{I}^*\times\mathcal{O}\}.
\end{equation}
\end{prop}

The following statement shows that, assuming that the perturbation is
Hamiltonian as in \eqref{eqn:h}, if the reduced Melnikov potential
additionally satisfies the non-degeneracy assumption \textbf{H4}, then
there exists homoclinic orbits along which the action variable changes
by $O(\eps)$.

\begin{thm}\label{prop:melnikov_jump}
Assume that the system \eqref{generalperturbation}   satisfies the conditions \textbf{H1}, \textbf{H2},
\textbf{H3},   the perturbation is Hamiltonian  as in  \eqref{eqn:h}, and that $\tilde{\phi}^t_\eps$ is an extended flow.

Let $\eps_0>0$ and the homoclinic manifold
$\Gamma_\eps$ be as Proposition \ref{prop:melnikov_int}, for
$0<|\eps|<\eps_0$.  Then for each point
${\tilde{x}}\in\tilde\Gamma_\eps$ there exists a uniquely defined pair of
points ${\tilde x}^-,{\tilde x}^+\in\tilde\Lambda_\eps$, such that
\begin{equation}\label{homoclinic}
d(\tilde\phi^t_\eps({\tilde x}),\tilde\phi^t_\eps({\tilde x}^\pm))<C_1
e^{\lambda_+ |t|}\, \textrm{ for } t\to\pm\infty,
\end{equation}
for some constant $C_1$  independent of $\epsilon$.

We also have that
\begin{equation}\label{eqn:Delta_I_theta}\begin{split}
I(\tilde{x}^+)  - I(\tilde{x}^-)  &= -\eps \partial_\theta {\widetilde{\mathcal{M}}^*}(I,\varphi-t\omega(I)) + O(\eps^2),
\end{split}\end{equation}
where $\tilde{x}^-=(I,\phi,t)$.

Hence, in particular, if we have \textbf{H4},  then we conclude
that for an open set of perturbations
 we have $|I(\tilde{x}^+) - I(\tilde{x}^-) | \ge C_3 |\eps|$,
for some constant $C_3$  independent of $\epsilon$.
\end{thm}

We point out that, in the above, the points $\tilde{x}^+$, $\tilde{x}^-$ depend
on the  locally invariant manifolds
considered, but the change in $I$- and $\theta$-coordinates from $\tilde{x}^-$ to $\tilde{x}^+$, as in \eqref{eqn:Delta_I_theta}, does not.

We note that the assumptions  of the previous results are very concrete
and can be verified in concrete systems of interest. In the following
result, we show that, when the perturbation is Hamiltonian and periodic,
the above assumptions are $C^k$-generic (indeed somewhat more
than that).

\begin{prop}\label{prop:genericity}
Assume that the perturbation is Hamiltonian  as in  \eqref{eqn:h}, and periodic, i.e., $t\in\torus$.
Then, each of the  non-degeneracy conditions \textbf{H3} and \textbf{H4} on the
Melnikov potential is $C^3$-open and $C^\infty$-dense for $h$.
\end{prop}

In Section~\ref{sec:scatteringmap},  we will describe an
alternative approach  to Theorem~\ref{prop:melnikov_jump}, based
on studying the scattering map.  The scattering map
$\tilde{S}_\eps : H_- \subset \tilde{\Lambda}_\eps \rightarrow \tilde{\Lambda}_\eps$
is defined  \cite{DelshamsLS00}
as $\tilde{S}_\eps(\tilde{x}^-) = \tilde{x}^+$, where $x^\pm$ are in
Theorem~\ref{prop:melnikov_jump}. The scattering map -- which
depends strongly on the $\Gamma_\eps$ chosen has remarkable
properties. In particular,   is  symplectic (when defined on the augmented symplectic phase space $\tilde{M}\times\real$, with symplectic form $\Upsilon+dA\wedge dt $), and  depends  smoothly
on parameters (in a suitable sense, since it defined on
a different manifold for every $\eps$).
Hence its first order effect is described
by a Hamiltonian.  It is quite remarkable that the
Hamiltonian describing in first order the effect of
the scattering map is precisely the Melnikov function
\eqref{eqn:melnikov_int}.


\begin{rem} \label{pseudorbits}
Because of \eqref{homoclinic}, we
note that, for large enough $T$,\break
$\{\tilde{\phi}^t_\eps(\tilde {x}^-)\}_{t \in ( -\infty, T']} $ and
$\{\tilde{\phi}^t_\eps(\tilde {x}^+)\}_{t \in [ T, \infty )} $
are very close to an orbit of  the full system, i.e.,
the orbit of the homoclinic of $\tilde{x}$.

We can consider that the orbits determined
by the $\{\tilde{\phi}^{-T}_\eps(x^-)\}_{t \in ( -\infty, T']}$ in the past and
by the $\{\tilde{\phi}^{T}_\eps(x^+)\}_{t \in [ T, \infty)}$  are pseudo-orbits of the
full system even if they are orbits in $\Lambda_\eps$.

Hence, we can consider  compositions of maps of the form
$F(\tilde{x}) = \tilde{\phi}^{T}_\eps\circ \tilde{S}_\eps \circ \tilde{\phi}^{T'}_\eps(\tilde{x}) $
as pseudo-orbits of the full system, even if the maps considered
are defined only on the NHIM.
Indeed, the main result of \cite{GideaLlaveSeara14} is a shadowing theorem
for pseudo-orbits of this type. If $\tilde\phi^t$ has good recurrence
properties, we can choose $T$, $T'$ so that the pseudo-orbits generated by maps as above are shadowed by true orbits of the full system. See \cite{GideaLlaveSeara14}.

There are some advantages to this construction. One is that we can
analyze orbits of the full system by analyzing pseudo-orbits in the
invariant manifold. More importantly, we can use several scattering
maps to produce pseudo-orbits.
\end{rem}

\subsection{Comparison  with  related works}

The Melnikov method takes its name from \cite{Melnikov63}
which considered the effect of perturbations on families of periodic orbits.
Some precedents can be found  in \cite[Volume 3]{Poincare99}.

One of the first papers which applied the Melnikov method to
perturbations of integrable systems is
\cite{HolmesM82}; the main example treated there is that of a pendulum
coupled with several oscillators subject to a periodic
perturbation.  The paper
\cite{HolmesM82} includes (not explicitly)  the
assumption that the perturbation
vanishes on the periodic orbit.   The paper  \cite{Robinson88} points
out that the Melnikov integrals in \cite{HolmesM82} are only
conditionally convergent {unless the perturbation
vanishes on the periodic orbit considered since
that the treatment of \cite{HolmesM82} does
not take into account the changes on the hyperbolic orbit induced
by the perturbation and assumes implicitly that the
perturbation vanishes on  the periodic orbit.
Furthermore,
\cite{Robinson88} shows  how to change the limits of
integration so that the result has the correct dynamic meaning.
The method of \cite{Robinson88}, works only for
the one degree of freedom case.

The well known books \cite{Wiggins90, Wiggins93} derive a
Melnikov theory for quasi-periodic orbits and quasi-periodic
perturbations by adopting a geometric point of
view.   The book  \cite{Wiggins90} does
not include explicitly the assumption that the
perturbation vanishes on quasi-periodic orbits,
but it realizes that, if they do not,
the main term in the formula is indefinite integrals
with quasiperiodic integrands
\cite[(4.2.85)]{Wiggins90}.  In \cite[p. 454]{Wiggins90}
it is argued that the integrals
do converge if one considers subsequences of times.
Nevertheless the zeroes of the Melnikov integral obtained depend on the arbitrary choice
of those sequences,  which is  difficult to give a concrete physical
interpretation.

The Melnikov theory for
quasi-periodic orbits under general
perturbations has the extra  difficulty that
 if the perturbation does
not vanish on them, there is no reason why they persist (as quasi-periodic
orbits).   In this
paper we obviate this problem by considering normally hyperbolic
manifolds which do persist. Even if the perturbation changes the
nature of the asymptotic orbits, we can still find the
asymptotic orbits and quantify the change. Of course, to obtain
correct results, one
still needs to take into account the fact that the normally hyperbolic
manifold itself changes.

A Melnikov potential whose critical points give rise to transverse homoclinic intersections in perturbation of periodic points
is introduced in \cite{DelshamsRamirezRos1996,DelshamsRamirezRos1997}.
In \cite{DelshamsG00,DelshamsG00a} the Melnikov potential
is established in a situation
where the hyperbolic part consists of a single pendulum.
A geometric version of the Melnikov integrals for periodic points
with codimension one  manifolds is studied in \cite{LomeliMR08}.
The paper\cite{LomeliMR08} assumes  explicitly
that the perturbations vanish on
the manifold.

Some other related papers  are
\cite{BaldomaFontich1998,LomeliMeiss2000,DelshamsLS00,DelshamsLS03a,DelshamsLS06b,Roy2006,LomeliMR08,DLS_Multi,GideaLlaveSeara14}.

\section{Proofs of the results}\label{section:theproofs}
\subsection{Proof of Proposition \ref{prop:nhim}}\label{sec:reminder}

The goal of this section is to apply some well known facts
from the theory of normally hyperbolic invariant manifolds
(persistence of invariant objects as well as some geometric
regularity properties) to analyze some of
the geometric objects present in the dynamics of the system
described by \eqref{generalperturbation}. The
structure of our problem makes it possible to obtain somewhat
sharper results than those afforded by the general theory. We will
also introduce some parametrization maps and systems of
coordinates.

\subsubsection{Persistence of the normally hyperbolic invariant manifold}
\label{subsection:persistence}
The perturbed equations are non-autonomous.
It is standard to make them
autonomous by considering an extra variable. A
precise way to do this is by renaming the independent variable to
$\sigma$ and adding an extra equation $\dot t=1$, where the
derivative is taken with respect to the time $\sigma$.
In the case of general perturbations take as the  augmented system
\begin{equation}\label{generalperturbation_t}
\begin{split}
& \dot x = \X^0(x) + \eps \X^1(x,t;\eps), \\
& \dot t  = 1.\\
\end{split}
\end{equation}
In the Hamiltonian perturbation case, we obtain:
\begin{eqnarray}\label{eqn:hamilton} \displaystyle \dot
p =&\displaystyle-\frac{\partial P}{\partial q}
-\eps\frac{\partial h}{\partial q}, \quad \displaystyle & \dot q
=\frac{\partial P}{\partial p}
+\eps\frac{\partial h}{\partial p},\\
\nonumber \displaystyle \dot I  =&\displaystyle-\eps\frac{\partial
h}{\partial \vphi}, \displaystyle\quad & \dot \vphi =\frac{\partial
h_0}{\partial I} +\eps\frac{\partial h}{\partial
I},\\\nonumber\displaystyle \dot t =&\displaystyle 1.  &
\end{eqnarray}

We  will always denote the independent variable by $\sigma$ from now
on. We will denote by $\tilde{\phi}^\sigma_\eps$ the flow on $\tilde
M=\mathbb{R}^n\times\mathbb{T}^n\times
\mathbb{R}^{d}\times\mathbb{T}^{d}\times \real$  generated by this
system. Since by \textbf{H1} we have that $V_i$, $h_0$ are
$C^{k+1}$-smooth and $\X^1$, in the general case, is
$C^{k}$-smooth (resp, $h$, in the Hamiltonian case, is $C^{k+1}$-smooth), with $k\geq 2$, the right-hand sides of the
equations \eqref{generalperturbation_t} and \eqref{eqn:hamilton} are $C^{k}$-smooth.

For the unperturbed system corresponding to $\eps=0$, we consider
the manifold with boundary
\[
\tLambda_0 = \{p = 0,\, q = 0,\, I \in \textrm{cl}(\mathcal{I}),\, \vphi \in
\torus^d,\, t \in \real\}.
\]
We notice that this manifold is invariant under the flow
$\tilde{\phi}^\sigma_0$, and the restriction of  $\tilde{\phi}^\sigma_0$ to
$\tLambda_0$ is  the form
\[\tilde{\phi}^\sigma_0(0,0,I,\vphi,t)=(0,0,I,\vphi+\omega(I)\sigma,t+\sigma).\]
Note that all the motions are either periodic or quasi-periodic
(depending on the rationality or not of $\omega(I)$). Furthermore,
the characteristic exponents for these motions are $0$.

In the $(p,q)$ variables, the motion is given by the $n$-penduli. By
\textbf{H2}, the point $(0,0)$ is a hyperbolic fixed point and there is
a homoclinic orbit $(p_i(t), q_i(t))$ to $(0,0)$. The characteristic
exponents are \begin{equation*}\label{lambda_i}\tilde\lambda_i=-(-V_i''(0))^{1/2}\textrm { and }
\tilde\mu_i=+(-V_i''(0))^{1/2},\, i=1, \ldots,n.\end{equation*}
Note that, in this case (as it happens in Hamiltonian systems)
the forward exponents $\lambda_i$ are equal to the backwards exponents).
This symmetry disappears when we consider non-Hamiltonian perturbations.

Thus, the global dynamics in the unperturbed
case  is the product of the quasi-periodic
motion on $\tLambda_0$ and the hyperbolic dynamics in the $(p,q)$
variables. We conclude that $\tLambda_0$ is a normally hyperbolic
manifold for $\tilde{\phi}^\sigma_0$. At each point $x\in\tilde\Lambda_0$,
the corresponding stable and unstable spaces $\tilde E^\st_x$ and
$\tilde E^\un_x$, respectively, are
\[\begin{split}\tilde E^\st_x=&\textrm{Span}\{(-(-V''_1(0))^{1/2},\ldots,
-(-V''_n(0))^{1/2},1,\ldots,1,0,0,\ldots, 0)\}, \textrm { and }\\
\tilde E^\un_x=&\textrm{Span}\{(-V''_1(0))^{1/2},\ldots,
(-V''_n(0))^{1/2},1,\ldots,1,0,0,\ldots, 0)\}.\end{split}\]

Invoking the theory
of normally hyperbolic invariant manifolds,
\cite{Fenichel71,HirschPS77,Pesin04}, we conclude that for all
 $|\eps| \ll 1$, there exist  normally hyperbolic manifolds $\tLambda_\eps$,
locally invariant for the flow $\tilde \phi^\sigma_\eps$, and $\Lambda_\eps$, locally
invariant  for the flow
$\tilde \phi^t_\eps$.  As it is well known, the way to construct
these locally invariant manifolds is by extending the vector fields  and
 showing that the extended flows corresponded to those vector fields have some
invariant manifolds. These invariant manifolds for the extended flow
are locally invariant for the original system. (We recall, however
that the invariant manifolds produced can depend on the extension
used.) In our case, the extensions are rather easy to perform
explicitly.
It suffices to extend $\X^0$ and $\X^1$ so that they remain
uniformly differentiable. See Appendix~\ref{sec:hyperbolicity}.

We will
assume henceforth that these extensions have been made and that we
are working in bounded domains on the extended manifolds. The extended manifolds are
normally hyperbolic and the hyperbolicity parameters are close to
those of the unperturbed manifold.


\begin{rem}
Since the normally hyperbolic expansion rates are
bounded uniformly for small $\eps$,  the expansion rates on
the invariant manifolds are close to $0$,
and the invariant manifolds are as smooth as the flow or the map.

More specifically,
 we can choose and fix positive constants
$\tilde\lambda_-$, $\tilde\lambda_+$, $\tilde\lambda_c$, $\tilde\mu_c$, $\tilde\mu_-$, $\tilde\mu_+$,
such that if
$\eps\in[-\eps_0,\eps_0]$ we have that $\tilde{\Lambda}_\eps$ is normally
hyperbolic and the following growth conditions are satisfied:
\begin{equation}\label{characterization}
\begin{split}
\|D\tilde{\phi}^\sigma_\eps(x)(v)\|&\leq C e^{\sigma\tilde\lambda_+}\|v\|
\textrm{ for all }
v\in E^\st_x,\, \sigma>0,\\
\|D\tilde{\phi}^\sigma_\eps(x)(v)\|&\leq C e^{\sigma\tilde\mu_-}\|v\|
\textrm{ for all } v\in E^\un_x,\, \sigma<0. \\
\|D\tilde{\phi}^\sigma_\eps(x)(v)\|&\leq C e^{\sigma\tilde\lambda_c}\|v\|
\textrm{ for all }
v\in T_x \tilde \Lambda_\eps,\, \sigma>0,\\
\|D\tilde{\phi}^\sigma_\eps(x)(v)\|&\leq C e^{\sigma\tilde\mu_c}\|v\|
\textrm{ for all }
v\in T_x \tilde \Lambda_\eps,\, \sigma< 0,\\
\end{split}
\end{equation}
where the $\lambda_+, \mu_+, \lambda_c, \mu_c$ are as close
as desired to the unperturbed ones by considering $|\eps|$
small enough. To avoid cluttering the notation, we
will choose them independently of $\eps$.
As a matter of fact, the converse of \eqref{characterization}
is also true, if a growth of a vector satisfies the
inequalities assumed in \eqref{characterization}, it is
in the corresponding space.

The fact that $\tilde\mu_c$ and $\tilde\lambda_c$ can be taken
close to $0$ implies that $\tilde\Lambda_\eps$ is $C^k$
with $C^k$ arbitrarily large for small enough $\eps$.

\end{rem}

\begin{rem}
In the symplectic case, it is natural to consider manifolds
for which
$\lambda_+ = -\mu_+$, $\lambda_c = - \mu_c$
(there is a pairing rule saying that for each rate the opposite
one appears). The manifolds that have these pairing rule
satisfy several geometric properties. On the other hand, we
note that there are other normally hyperbolic manifolds which
do not satisfy the pairing rule. For example, the stable manifold
of a NHIM is also a NHIM.
\end{rem}

\begin{rem}
The manifold $\tLambda_\eps$ is not symplectic -- it has
odd dimension  -- but if we fix the coordinate $t=\sigma_0$, we
obtain an exact symplectic manifold $\Lambda_\eps^{\sigma_0}$ with symplectic form $\Upsilon^{\sigma_0}_\eps$.
(Note that if $\Upsilon$ is the symplectic
form in the phase space,  $d|_{\Upsilon_\eps} \Upsilon|_{\Lambda_\eps} =
\left(d \Upsilon\right)|_{\Lambda_\eps} = 0$ and that the nondegeneracy
of $\Upsilon|_{\Lambda_\eps}$ is an open condition. If $\Upsilon = d\alpha$,
$\Upsilon|_{\Lambda_\eps}  = d|_{\Lambda_\eps} \alpha|_{\Lambda_\eps}$.

When the perturbations are Hamiltonian,
the flow $\tilde{\phi}^\sigma_\eps$ maps $\Lambda^{\sigma_0}_\eps$ to
$\Lambda^{\sigma+\sigma_0}_\eps$, and is symplectic in the sense that it preserves the symplectic structure, i.e.,
$(\tilde{\phi}^\sigma_\eps)^*(\Upsilon^{\sigma+\sigma_0}_\eps)=\Upsilon^{\sigma_0}_\eps$.
\end{rem}

Attached to the invariant manifold  $\tLambda_\eps$ there exist stable and unstable manifolds $W^{\st}(\tLambda_\eps)$ and, respectively
$W^{\un}(\tLambda_\eps)$.
These manifolds are, in our case,  $C^k$-immersed (due to the center bunching
conditions)
manifolds. For details, see Appendix \ref{sec:extension}.

\subsection{Proof of Proposition \ref{prop:melnikov_int}.}
\label{sec:Melnikov}

In this section we motivate the Melnikov vector and potential given by \eqref{eqn:melnikov_vect} and \eqref{eqn:melnikov_int}, and we prove Proposition \ref{prop:melnikov_int} on the convergence of the integrals involved in the definition of the  Melnikov vector and potential.
As we will see, the main tools are just the results on
smooth dependence on parameters from
the theory of normally hyperbolic invariant manifolds,
the results on smooth dependence on
initial conditions for ordinary differential
equations and  the fundamental theorem of calculus.

The main idea of this  paper is that we do not try to establish the
derivatives of the stable/unstable manifolds with respect to the
perturbation parameter, but to obtain formulas for these derivatives
(which we know they exist by the general theory).
The calculation just takes advantage of the fact that the
variables $P$ are slow (evolve at a speed $\eps$
so that one can use the fundamental
theorem of calculus  over a scale of time much smaller than $\eps^{-1}$.
This calculation applies to general perturbations and does
not use Hamiltonian structure.  In case that the perturbation is
Hamiltonian, one can get several extra properties and the
perturbation can be expressed as a gradient.

\subsubsection{Convergence properties of the integrands in  \eqref{eqn:melnikov_vect} and \eqref{eqn:melnikov_int} }

To prove   the first claim in Proposition~\ref{prop:melnikov_int},
namely that the
integrands in \eqref{eqn:melnikov_vect} and
\eqref{eqn:melnikov_int} converge  uniformly
together with the derivatives, we  start with some general preparation.
This is based on some purely real analysis
results; similar estimates appear in \cite{LlaveMM86,BanyagaLW96,LlaveW10}.

In the calculations that we will carry out in Section
\ref{sec:graphs}, we will find it useful to use the notation of wave
maps introduced in Section~\ref{sec:wavemaps}.  That is,
${\tilde  x}^\pm \equiv \Omega_\eps^{\st,\un}(\tilde x)$ denotes the point in
$\tilde{\Lambda}_\eps$ such that the orbit of $\tilde x$ is
exponentially close (with a fast enough rate) in the future to the
orbit of ${\tilde x}^\pm$.  We refer to Section~\ref{sec:wavemaps} for
more details on the properties of the mappings
$\Omega_\eps^{\st,\un}$.

Taking into account \eqref{convergences}, \eqref{convergenceu} from  Appendix \ref{sec:scatteringnhim}, we see that  when $T
\rightarrow \pm\infty$, we have that $P_i( \tilde{\phi}^T_\eps({\tilde x})) - P_i(
\tilde{\phi}^T_\eps \tilde{\Omega}_\eps^{\st,\un}({\tilde x}) )$ converges to $0$ exponentially
fast.

The main observation, which will be used in Section \ref{sec:graphs}, is that, for
every $\tilde{x}$ in a neighborhood of $\tilde{\Lambda}_\eps$ in
$W^{\st,\un}(\tilde{\Lambda}_\eps)$,
 there is a path $\gamma^{\st,\un}_{\tilde{x}}$ contained in
$W^{\st,\un}_\eps( \tilde{\Omega}^{\st,\un}_\eps(\tilde{x}))$ such that
$\gamma^{\st,\un}_{\tilde x} (0) = \tilde{\Omega}^{\st,\un}_\eps(\tilde x))$,
$\gamma^{\st,\un}_{\tilde x}(1) = \tilde x$.
We can also ensure that
\begin{equation}
\label{uniformity}
|D_{\tilde x}^l \gamma^{\st,\un}_{\tilde x} (\sigma) | \le C
\end{equation}
for all $\tilde x$ in a bounded neighborhood of
$\tilde{\Lambda}_\eps$ in $W^{\st,\un}_\eps(\tilde{\Lambda}_\eps)$.  We can assume
 that the neighborhood where these paths are
obtained is forward (resp., backwards) invariant.

By the fundamental theorem  of calculus and the chain rule,
given a smooth function $f$,
\begin{equation}\label{interpolation}
\begin{split}
f( \tilde{\phi}_\eps^T & (\tilde x) ) -
f( \tilde{\phi}_\eps^T( \tilde{\Omega}^{\st,\un}_\eps(\tilde x)) )
=  \int_{0}^1 \frac{d}{d \sigma} f( \tilde{\phi}_\eps^T
 \gamma^{\st,\un}_{\tilde x}(\sigma))\, d \sigma \\
=& \int_{0}^1 [Df]( \tilde{\phi}_\eps^T \gamma^{\st,\un}_{\tilde x}(\sigma))  \,
(D \tilde{\phi}_\eps^T)( \gamma^{\st,\un}_{\tilde x}(\sigma))  \,
\left(\frac{d}{d \sigma} \gamma^{\st,\un}_{\tilde x}\right) (\sigma) \, d \sigma.
\end{split}
\end{equation}

To prove the  exponential convergence of
the integrands in  \eqref{eqn:melnikov_vect} and \eqref{eqn:melnikov_int},  we will need
to show, as it will be seen in Section \ref{sec:graphs},  that  the integrand in \eqref{interpolation}
decreases exponentially fast with $T$.

The key observation is that, we have
\cite[p.  574]{LlaveMM86}, \cite{LlaveW10},
\[
\begin{split}
|D^l_{\tilde{x}}\tilde{\phi}^T_\eps(\tilde{x}) | \le C_l T^{l+1} e^{(l\mu_c + \lambda_+) T}
\end{split}
\]
So that if $l <  -\lambda_+/\mu_c$,
 the derivatives of the flow converge exponentially
fast in $T$.  Note that for $|\eps|$ sufficiently small,
$\lambda_+$ is arbitrarily to a fixed number
as in \textbf{H2} and $\mu_c$ is close to $0$, so
that we obtain $l$ as large as desired.

If we apply the Faa-di-Bruno formula, to
compute the $l$-th  derivative of the integrand in
\eqref{interpolation}, we obtain several terms whose
 factors which are
are either derivatives of $\tilde{\phi}^T_\eps$ -- hence
 exponentially convergent with a total
exponent $ e^{(l\mu_c +  \lambda_+) T} $ times a polynomial in $T$
or uniformly bounded.
}

\subsubsection{A system of coordinates in
a neighborhood of the homoclinic connection}
\label{sec:coordinates}

We start by observing that \textbf{H2} implies that for the dynamics
of the $i$-th pendulum, in the variables $p_i, q_i$, the point
$(0,0)$ is an isolated critical point of the Hamiltonian $P_i$.
The zero level set of the Hamiltonian $P_i$ -- which we will
denote by $\Sigma_i$
 -- is a curve with singularities  only at $(0,0)$ and otherwise as smooth as
the Hamiltonian. This level set contains at least one orbit -- often
$2$ -- homoclinic to $(0,0)$.

If we  consider the product system of the $n$ penduli, we note
that the origin is a hyperbolic critical point. The product
$\Sigma = \Pi_{i=1} ^{n} \Sigma_i$ is the set of orbits homoclinic
to the origin in $\mathbb{R}^n\times \mathbb{T}^n$.

For each of the indices $i$, we select one of the homoclinic
orbits implied to exist by \textbf{H2}, and choose an origin of the
parameter.  We write
\begin{equation}\label{parametrization}
(p^0_i(\tau_i) , q_i^0(\tau_i) )
\end{equation}
for the selected homoclinic orbit to $(0,0)$, parameterized by its
natural time. The origin of time is chosen in such a way that
$p^0_i(0) = p^*_i, q^0_i(0) = q^*_i$, where $(q^*_i, p^*_i) \ne (0,0)$
are chosen once and for all.

We note that, as $(\tau_1,\ldots ,\tau_n)$ ranges over $\real^n$,
the function
\[
(p^0_1(\tau_1),\ldots ,p^0_n(\tau_n),
q_1^0(\tau_1),\ldots,q_n^0(\tau_n))
\]
gives a parametrization of a connected component  of the homoclinic manifold
$\Sigma$ in the unperturbed case. We will denote it by $\Sigma^*$. In the case when
 the critical point of each pendulum has $2$ homoclinic
orbits, $\Sigma$ will have $2^n$ components as above corresponding to the
different choices of the homoclinic orbits.

Choose $L$ a sufficiently large number and consider a subset of
the $\mathbb{R}^n\times\mathbb{T}^n$ (corresponding to the $(p,q)$
coordinates):
\[\widehat\Sigma  = \{(p^0_1(\tau_1),\ldots ,p^0_n(\tau_n),
q_1^0(\tau_1),\ldots,q_n^0(\tau_n))\, |\, |\tau_i| \le L \textrm {
for all } i=1,\ldots,n \}.\]
The conditions on $L$ will be made explicit in the calculations, but
we emphasize that are independent of $\eps$.
Let $\widehat\Sigma^\eta$ be an $\eta$-neighborhood of $\widehat\Sigma$ in
$\mathbb{R}^n\times\mathbb{T}^n$, for some $\eta>0$ sufficiently
small.

We introduce a new coordinates system on $\widehat\Sigma^\eta$ which is a
product of coordinate systems in the pendulum spaces
$\{(p_i,q_i)\}$, as we describe below.

For each pendulum, the first coordinate of a point $(p_i,q_i)$  is
the corresponding value of the Hamiltonian (energy level)
$P_i(p_i,q_i)=\pm\left[\frac{1}{2}p_i^2+V(q_i)\right]$. The second
coordinate $\tau_i=\tau_i(p_i,q_i)$ is defined as  the time
time $\tau$ of the closest point to $(p_i,q_i)$ on
the orbit $\Sigma^*_i$, that is,
\[
d( (p^0_i(\tau_i) , q_i^0(\tau_i) ), (p_i,q_i))
= \min_{|\tau|\leq L} d((p^0_i(\tau) , q_i^0(\tau))  , (p_i,q_i)).
\]

By the implicit function theorem, $\tau_i$ is
defined in a small enough neighborhood of the point $(p_i,q_i)$,
and  is an extension of the time defined on the homoclinic orbit
(this is why we use the same letter).

Note that the gradient of
$\tau_i$ along the homoclinic connection is tangent to the level
set $P_i=0$. This gradient is bounded away from zero on the
compact set $\widehat\Sigma^\eta\cap\{(p_i,q_i)\}$. The gradient of $P_i$ is perpendicular to the
homoclinic connection. Hence, by the implicit function theorem,
the functions $(P_i,\tau_i)$ define a system of coordinates in a
neighborhood of the homoclinic connection restricted to
$\widehat\Sigma^\eta \equiv \{ |\tau|\leq L$,\, $|P_i| \leq \eta \ll 1 \}$.
Therefore, $(P_i,\tau_i)_{i=1,\ldots,n}$ define a
system of coordinates for the penduli on  $\widehat\Sigma^\eta$, for $\eta$
sufficiently small.   Note that this coordinate system does not
extend to a coordinate system near the equilibrium points of
each of the penduli. At the equilibrium point of
the $i$-th pendulum, $P_i$ has a critical point and,
therefore, is not a good coordinate.

We will follow the standard practice in dynamics to denote
a point in phase space by the coordinates. The letters used
will specify which system of coordinate is being used. In cases
where this could lead to confusion, we will use a more precise
notation specifying the ranges and domains explicitly.

The  coordinate system  $(P_i,\tau_i)_{i=1,\ldots,n}$,
corresponding to the penduli, together with the remaining
coordinates $(I,\vphi)$, corresponding to the rotators, constitute a
system of coordinates on the phase space in a neighborhood of the
homoclinic connection.

In summary, we will use the coordinate system  $(P,\tau,I,\vphi,t)$
in the domain $\widehat\Sigma^\eta\times
\mathbb{R}^d\times\mathbb{T}^d\times\mathbb{R}$ of the augmented phase space
$\mathbb{R}^n\times\mathbb{T}^n\times\mathbb{R}^d\times\mathbb{T}^d\times\mathbb{R}$, where $P=(P_i)$ and $\vphi=(\vphi_i)$.
Each point ${\tilde x}$
in this domain can be specified  as ${\tilde x}\equiv (P,\tau,I,\vphi, t)$,
and also as ${\tilde x}\equiv (p,q,I,\vphi, t)$;

This coordinate system has several
properties that we will find quite useful for later applications,
which we note for future reference:
\begin{itemize}
\item The variables $P_i$ extend to the whole space of the penduli
$\mathbb{R}^n\times\mathbb{T}^n$;

\item The $P_i$- and $I$-coordinates of points are preserved by the flow
$\tilde{\phi}^\sigma_0$ for $\eps = 0$ and, therefore, for $\eps$ small
enough, they will be slow variables; on the other hand the variables $\tau,\vphi,t$ move with speed $O(1)$, so they will be the fast variables.

\item For $\eps = 0$, the manifolds
$W^{\st,\un}(\tilde\Lambda_0)$ have
 very simple expressions since they are just obtained by setting
$P_i = 0$ and letting $\tau_i, I, \vphi, t$ vary.
\end{itemize}

\subsubsection{Representation of the stable/unstable manifolds as
graphs}\label{sec:graphs}

We will work in the augmented phase space
$\mathbb{R}^n\times\mathbb{T}^n\times
\mathbb{R}^d\times\mathbb{T}^d\times\real$. For $\eps = 0$, the
stable and unstable manifolds are given by setting all the $P_i$
coordinates  equal to $0$. That is, the portion of the stable manifolds
we are interested in is the graph of the function $0$ from the
$(\tau, I, \vphi, t)$ coordinates  to the $P$ coordinates.
The variable $\tau$ will range over a bounded domain, which
we will take to be $[ -c | \ln|\eps||, c | \ln|\eps|]$.
Note that it this is a good system of coordinates, even if
for large values of $|\tau|$ it is moderately singular.

Since the
derivative of the zero function is bounded, the smooth dependence of
the stable manifolds on parameters that one obtains from the
standard theory of normally hyperbolic invariant manifolds implies
that, for every $\eps$ sufficiently small there is a function
$\Psi^\st_\eps$ that assigns to each $(\tau, I, \vphi, t)$, with
$|\tau_i|\leq L$ for all $i$, a unique $P=\Psi^\st_\eps(\tau, I, \vphi,t) \in \real^n$ such that
\[ \left(\Psi^\st_\eps(\tau, I, \vphi,t),
\tau, I, \vphi, \right)\in W^\st(\tLambda_\eps).\] An analogous
argument holds for the unstable manifolds, providing us with a
function $\Psi^\un_\eps$ that assigns to each $(\tau, I, \vphi, t)$,
with $|\tau_i|\leq L$ for all $i$,  a unique $P=\Psi^\un_\eps(\tau, I, \vphi,t)\in \real^n$ such
that \[(\Psi^\un_\eps(\tau, I, \vphi,t), \tau, I, \vphi, t)\in W^\un(\tLambda_\eps).\] See
Figure \ref{fig_graphs}.

From the theory of normally hyperbolic
manifolds
\cite{HirschPS77, Fenichel71,Pesin04}.
it follows that the functions $\Psi^{\st,\un}_\eps$ are
jointly $C^{k}$ in all variables and parameters,  with all the
derivatives being uniformly bounded in the domain of definition of
the functions $\Psi^{\st,\un}_\eps$. See Appendix \ref{sec:extension}.

\begin{figure}
\includegraphics[width=0.5\textwidth]{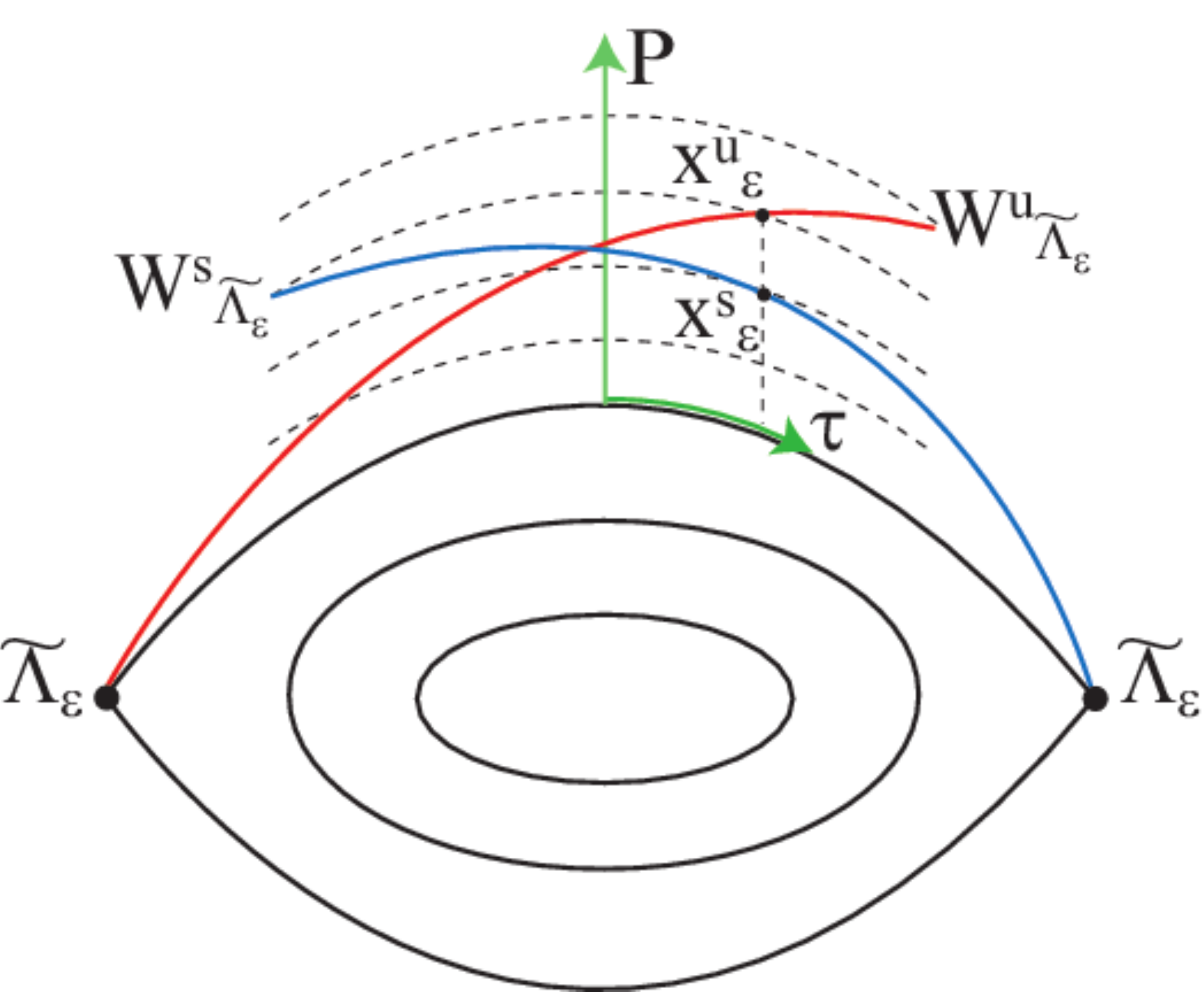}
\caption{The stable and unstable manifolds as graphs.}
\label{fig_graphs}
\end{figure}

\subsubsection{The perturbation equations}
We denote,  the Poisson bracket by $[\cdot, \cdot]$, we have
that $[ P_i, P_j ] = 0$, $[P_i, h_0] = 0$, $[\tau_i, h_0]=0$, $[I,
h_0] = 0$, $[I, P_i] = 0$, $[\vphi,h_0]=\partial h_0/\partial I$,
$[\vphi,h]=\partial h/\partial I$.

We can write the evolution
equations of the coordinates introduced in
Section~\ref{sec:coordinates} as, with the last line in each equation corresponding to the Hamiltonian perturbation:
\begin{equation}\label{evolution}
\begin{split}
\dot P_i &= (\X^0+\eps X^1)(P_i) \\ & = \eps\X^1(P_i)  \\ & = \eps[ P_i,  h], \\
\dot \tau_i &= (\X^0+\eps X^1)(\tau_i)\\&=[\tau_i,P_i]+\eps\X^1(\tau_i) \\ &=[\tau_i,P_i]+\eps[\tau_i,h],\\
\dot I &= (\X^0+\eps X^1)(I) \\&= \eps X^1(I) \\&= \eps[ I,  h], \\
\dot \vphi &=  (\X^0+\eps X^1)(\vphi)\\&=\frac{\partial h_0}{\partial I}
+ \eps X^1(\vphi)\\&=\frac{\partial h_0}{\partial I}
+ \eps\frac{\partial h}{\partial I}.
\end{split}
\end{equation}

Note that the variables $P_i, I$ are slow because the right hand
side of \eqref{evolution} has an $\eps$ factor.  On the other hand,
the variables $\tau_i$, $\vphi$ have evolutions which are bounded
away  from zero uniformly as $\eps$ becomes small. Recall that we
made the Hamiltonian system autonomous by adding $t$ as an extra
variable satisfying the additional equation
\[
\dot t = 1.
\]
Note that this makes $t$ a
 fast variable since  that the right hand side does
not approach $0$ as $\eps$ approaches $0$.

\subsubsection{The perturbative formula for the invariant graphs}

Now we derive explicitly the integrals that appear in \eqref{eqn:melnikov_vect} and \eqref{eqn:melnikov_int}.
We will simultaneously treat the general perturbation case, and the Hamiltonian perturbation case, emphasizing the differences when appropriate.

By the considerations in Subsection \ref{sec:graphs}, we know
that the implicit function theorem implies that
we can write a piece of
the stable manifold as the graph of a differentiable functions. (A similar statement holds for the unstable manifold.)
That is,
\[
\tilde{x}\equiv( P, \tau, I, \vphi, t) \in W^\st(\tLambda_\eps),
 \textrm{  with }
P =  \Psi^\st_\eps(\tau, I, \vphi,t),
\]
for an explicit function $\Psi^\st_\eps$.

The goal of this section is to
obtain the first order term in the expansion in $\eps$ of
$\Psi^\st_\eps$.
 The procedure we will follow is rather standard and it
goes back to \cite{Poincare99}. It amounts to applying the
fundamental theorem of calculus to relate the
value of $\Psi^\st_\eps$ at a certain point in $W^\st(\tilde{\Lambda}_\eps)$
to the value of $\Psi^\st_\eps$ at a  point in the orbit at
more advanced time.
Imposing the condition that the behavior at the point
advanced time converges, we obtain an integral expression
for the $\Psi^\st_\eps(\tau, I, \vphi, t)$ which will allow us to
compute the desired expansion. Note that this
uses the fact that $P$ is a slow variable and that the
convergence of the orbits is exponential.

Of course, once we have computed
an approximation for $\Psi^\st_\eps$, we can
obtain a similar expression for $\Psi^\un_\eps$,
by considering
inverse flow (or by performing directly a very similar argument).

For the moment, we will consider  a point $\tilde{x}\equiv(P,\tau, I, \vphi,
t)\in W^\st(\tLambda_\eps)$ fixed and we will derive an approximate
expression for $\Psi^\st_\eps(\tau, I, \vphi, t)$.

We will use the notation $\tilde \Omega^s_\eps(\tilde x)$ to denote
the only point in $\Lambda_\eps$ such that
$\tilde x \in {\tilde W}^s_{\tilde \Omega^s_\eps(\tilde x)}$.
See Section~\ref{sec:wavemaps} for more properties of these maps.

First, we will
estimate the error terms of the approximation of $\Psi^\st_\eps$ in
the $C^0$-norm. Later we will estimate them in smooth norms.

By the fundamental theorem of calculus, for any
point $\tilde x \in W^\st(\tilde{\Lambda}_\eps)$ and any  $T > 0$ we have
the following identity (the last line is, of course,
true only in the Hamiltonian case).

\begin{equation}\label{FTC}
\begin{split}
P_i(\tilde{x}) - P_i(\tilde{\Omega}^\st_\eps(\tilde{x}))
= & P_i( \tilde{\phi}^T_\eps(\tilde{x})) - P_i(\tilde{\phi}^T_\eps \tilde{\Omega}_\eps^\st(\tilde{x}) ) \\
& -\int_0^T  \frac{d}{d \sigma}\big[ P_i( \tilde{\phi}^\sigma_\eps(\tilde{x})) -
P_i( \tilde{\phi}^\sigma_\eps \tilde{\Omega}_\eps^\st(\tilde{x}) ) \big] \, d\sigma\\
= & P_i( \tilde{\phi}^T_\eps(\tilde{x})) - P_i(\tilde{\phi}^T_\eps \tilde{\Omega}_\eps^\st(\tilde{x}) ) \\
& -\eps\int_0^T \big[ (\X^1 P_i)( \tilde{\phi}^\sigma_\eps(\tilde{x})) -
(\X^1 P_i)( \tilde{\phi}^\sigma_\eps \tilde{\Omega}_\eps^\st(\tilde{x}) ) \big] \, d\sigma \\
= & P_i( \tilde{\phi}^T_\eps(\tilde{x})) - P_i(\tilde{\phi}^T_\eps \tilde{\Omega}_\eps^\st(\tilde{x}) ) \\
& -\eps\int_0^T  \big[ [P_i, h]( \tilde{\phi}^\sigma_\eps(\tilde{x})) -
[P_i,h]( \tilde{\phi}^\sigma_\eps \tilde{\Omega}_\eps^\st(\tilde{x}) ) \big] \, d\sigma.
\end{split}
\end{equation}

 Using that $P_i(\tilde{x}) - P_i(\tilde{\Omega}^\st_\eps(\tilde{x}))\to 0$ as $T\to \infty$, as shown by \eqref{interpolation},    making $T=c\log\left(\frac{1}{\eps}\right)$, and using the evolution equations  \eqref{evolution}  in \eqref{FTC},
we obtain (again, the last identity is true in the Hamiltonian case):
\begin{equation}\label{FTC2}
\begin{split}
P_i({\tilde x}) - P_i( \tilde{\Omega}^\st_\eps({\tilde x})) =&-\eps\int_0^{c
\log(\frac{1}{\eps})}
  \big[ (\X^1 P_i)( \tilde{\phi}^\sigma_\eps({\tilde x})) - (\X^1 P_i)( \tilde{\phi}^\sigma_\eps \tilde{\Omega}^\st_\eps({\tilde x}) )
\big] \, d\sigma \\&\,\,+ O_{C^{k}}(\eps^2) \\
=&-\eps\int_0^{c
\log(\frac{1}{\eps})}
  \big[ [P_i,h]( \tilde{\phi}^\sigma_\eps({\tilde x})) - [P_i,h]( \tilde{\phi}^\sigma_\eps \tilde{\Omega}^\st_\eps({\tilde x}) )
\big] \, d\sigma \\&\,\,+ O_{C^{k}}(\eps^2).
\end{split}
\end{equation}

Recall that the evolution equations \eqref{evolution} show that
$P_i$ is a slow variable. Taking into account that the integrand
depends smoothly on parameters, to compute the first order term we
only need to compute the integral over the unperturbed trajectory
corresponding to $\eps = 0$.

More precisely, we notice
that because of the smooth dependence on parameters of the
normally hyperbolic manifolds, ${\tilde x}$ can be expressed with respect to
the original variables as
\begin{equation}\label{eqn:x0} {\tilde x}={\tilde x}^0+O_{C^{k}}(\eps),\end{equation}
for some point ${\tilde x}^0\equiv(p^0(\tau),q^0(\tau), I,\vphi,t) \in
\Sigma^*$, the part of the homoclinic manifold selected earlier in
the section.

A key observation is that, since $\eps \log(|\eps|)$ is
much smaller than $1$, we can use Gronwall's inequality
to control the perturbations in $\eps$ for
$\tilde \phi^\sigma_\eps$ uniformly for $|\sigma| \le C |\log(|\eps|) |$ (see
\cite{Hale77,Hartman02}):

\begin{equation}
\label{perturbation}
\begin{split}\tilde{\phi}^\sigma_\eps({\tilde x})=\tilde{\phi}^\sigma_0({\tilde x}^0)+O_{C^{k}}(\eps),\\
\tilde{\phi}^\sigma_\eps(\tilde{\Omega}^\st_\eps({\tilde x})))=\tilde{\phi}^\sigma_0(\tilde{\Omega}^\st_0
({\tilde x}^0))+O_{C^{k}}(\eps),\\
\X^1({\tilde x};\eps)=\X^1({\tilde x}^0;0)+O_{C^{k}}(\eps),\\
h({\tilde x};\eps)=h({\tilde x}^0;0)+O_{C^{k+1}}(\eps).\end{split}
\end{equation}
where we emphasize that the $O_{C^k}(\eps)$ are uniform
in $|\sigma| \le C | \ln( |\eps|)|$, and the last line refers to the Hamiltonian perturbation.


Substituting \eqref{perturbation}  in \eqref{FTC2} gives us
\begin{equation}\label{goodformula}\begin{split}
P_i({\tilde x}) - P_i( \tilde{\Omega}^\st_\eps({\tilde x}))  = - &\eps
\int_0^{c\log(\frac{1}{\eps})} \left[(\X^1 P_i)(\tilde{\phi}^\sigma_0({\tilde x}^0))
-(\X^1 P_i)(\tilde{\phi}^\sigma_0(\tilde{\Omega}^\st_0({\tilde x}^0)) \right]{d\sigma}\\ &+
O_{C^{k}}(\eps^2\log\eps)+O_{C^{k}}(\eps^2)\\
= - &\eps
\int_0^{c\log(\frac{1}{\eps})} \left [[P_i, h](\tilde{\phi}^\sigma_0({\tilde x}^0))
-[P_i, h](\tilde{\phi}^\sigma_0(\tilde{\Omega}^\st_0({\tilde x}^0)) \right]{d\sigma}\\ &+
O_{C^{k}}(\eps^2\log\eps)+O_{C^{k}}(\eps^2).\end{split}
\end{equation}
Of course, the first equality is true in general and the
second one is valid only in the Hamiltonian case.

The remainder $O_{C^{k}}(\eps^2\log\eps)+O_{C^{k}}(\eps^2)$ is of order
$O_{C^{k}}(\eps^{1+\varrho})$ for any $0< \varrho<1$.

For the point ${\tilde x}\in W^\st(\tLambda_\eps)$, we now express
$P({\tilde x})=(P_1({\tilde x}),\ldots, P_n({\tilde x}))$ and the right-hand side of
\eqref{goodformula}  in terms of $(\tau,I,\vphi, t)$; however, we
will  express $\X^1$ in terms of the original variables
$(p,q,I,\vphi,t;\eps)$. We also let the integral run to $\infty$ by
estimating the remainder of the integral as a quantity of order
$O_{C^{k}}(\eps^{1+\varrho})$, yielding
\begin{equation}\label{bestformulaX}\begin{split}
&(\Psi^\st_\eps)_i(\tau,I,\vphi,t) - P_i(\tilde{\Omega}^\st_\eps({\tilde x}))\\
&\quad =-\eps
\int_{0}^{\infty} \left
[(\X^1P_i)(p^0(\tau+\sigma\bar 1),q^0(\tau+\sigma\bar 1),I,\vphi+\sigma\omega(I), t+\sigma;0)\right.\\
 &\quad\quad\qquad  \left. - (\X^1P_i)(0,0,I,\vphi+\sigma\omega(I), t+\sigma;0) \right]{d\sigma}\\
  &\quad\quad  + O_{C^{k}}(\eps^{1+\varrho}).
\end{split}
\end{equation}

The same argument ran for the reversed dynamics yields:
\begin{equation}\label{bestformula2X}\begin{split}
&(\Psi^\un_\eps)_i(\tau,I,\vphi,t) - P_i(\tilde{\Omega}^\un_\eps({\tilde x}))\\
&\quad =\eps
\int_{-\infty}^{0} \left
[(\X^1P_i)(p^0(\tau+\sigma\bar 1),q^0(\tau+\sigma\bar 1),I,\vphi+\sigma\omega(I), t+\sigma;0)\right.\\
 &\quad\quad\qquad  \left. - (\X^1P_i)(0,0,I,\vphi+\sigma\omega(I), t+\sigma;0) \right]{d\sigma}\\
  &\quad\quad  + O_{C^{k}}(\eps^{1+\varrho}).
\end{split}
\end{equation}

We finally observe that,
since $P$ has a critical point at $(0,0)$, $\|P(\tilde{\Omega}^\st_\eps({\tilde x^0})) -P(\tilde{\Omega}^\un_\eps({\tilde x^0}))\|_{C^{k}} \le C
\eps^2$.
This is because we have smooth dependence on parameters
for the NHIM and their invariant manifolds (as well as their
derivatives), we can also estimate the derivatives with respect to $\tilde{x}$. Due to
\eqref{eqn:x0}, we have that
 $\|P(\tilde{\Omega}^\st_\eps({\tilde x})) -P(\tilde{\Omega}^\un_\eps({\tilde x}))\|_{C^{k}}=O_{C^k}(\eps^2)$.

We conclude that
\begin{equation}\label{melnikov_v_eps}\begin{split} &(\Psi^\un_\eps)_i( \tau, I, \vphi, t) -(\Psi^\st_\eps)_i
( \tau, I, \vphi, t)
\\&\quad = \eps\int_{-\infty}^\infty
\, \left[(\X^1P_i)( p^0(\tau+\sigma\bar 1 ), q^0(\tau+\sigma\bar 1 ), I ,
\vphi + \sigma\omega(I), t+\sigma;0) \right.\\&
\quad\quad\qquad\left .- (\X^1P_i)( 0, 0, I , \vphi +  \sigma\omega(I),
t+\sigma;0)\right] d\sigma\\&\quad\quad+
O_{C^{k}}(\eps^{1+\varrho})
\\&\quad=\eps\Mv_i(\tau,I,\phi,t)+
O_{C^{k}}(\eps^{1+\varrho}).\end{split}
\end{equation}

In the Hamiltonian case we observe that writing
\[
\tilde{\phi}^\sigma_0({\tilde x}^0)
\equiv(p^0(\tau+\sigma\bar 1),q^0(\tau+\sigma\bar 1), I,\vphi+\sigma\omega(I), t+\sigma),\]
we have  that
\begin{equation}\label{tauderivative}
\begin{split}
[P_i, h]( \tilde{\phi}^\sigma_0({\tilde x}) )
 &= \left [\frac{\partial P_i}{\partial q_i}\frac{\partial h}{\partial p_i}-
 \frac{\partial P_i}{\partial p_i} \frac{\partial h}{\partial
 q_i}\right]( \tilde{\phi}^\sigma_0({\tilde x}^0) )\\&=-\frac{\partial}{\partial\tau_i}h( \tilde{\phi}^\sigma_0({\tilde x}^0)).\end{split}
\end{equation}

From \eqref{melnikov_v_eps}, using the fact that the derivative
with respect to $\tau_i$ and the integral with respect to $t$
commute, and applying \eqref{tauderivative},
 we obtain
\begin{eqnarray}\label{melnikoveps}\qquad
\lefteqn{\Psi^\un_\eps( \tau, I, \vphi, t) -\Psi^\st_\eps( \tau, I,
\vphi, t)}
\\\nonumber
&&\quad -\eps\frac{\partial }{\partial \tau} \int_{-\infty}^\infty
\, \left[h( p^0(\tau+\sigma\bar 1 ), q^0(\tau+\sigma\bar 1 ), I ,
\vphi + \sigma\omega(I), t+\sigma;0) \right.\\&& \nonumber \qquad
\qquad\qquad\left .- h( 0, 0, I , \vphi +  \sigma\omega(I),
t+\sigma;0)\right] d\sigma\\\nonumber &&\quad+
O_{C^{k}}(\eps^{1+\varrho}).
\end{eqnarray}

\subsubsection{The Melnikov Potential.}

We now want to measure the splitting of the stable and unstable
manifolds regarded as functions of $(\tau, I,\vphi, t)$. For this,
we pick a point ${\tilde x}^0\equiv(0, \tau,I,\vphi,t)$ on the homoclinic
manifold of the unperturbed system, and consider the corresponding
points
\begin{equation*}
\begin{split}
{\tilde x}^\st_\eps&\equiv(\Psi^\st_\eps(\tau,I,\vphi,\sigma),\tau,I,\vphi,t)\in
W^\st(\tLambda_\eps),\\
{\tilde x}^\un_\eps&\equiv(\Psi^\un_\eps(\tau,I,\vphi,\sigma),\tau,I,\vphi,t)\in
W^\un(\tLambda_\eps).
\end{split}
\end{equation*}
See Figure \ref{fig_graphs}.
The splitting of
the stable and unstable manifolds measured at one point on the
homoclinic manifold with respect to the $P$-variable  is given by
the distance given by \eqref{melnikoveps}.

\begin{rem}
The regularity of the reminders
can be easily bootstrapped.
We need to emphasize that the regularity of
a function taking values in a space, requires
specifying a topology in this space.
We recall
that the theory of normally hyperbolic invariant
manifolds establishes that the stable/unstable manifolds depend $C^2$
on parameters when the manifolds are given the $C^{k-2}$ topology.
When $k \ge 2$, the derivatives with respect to
parameters in the sense of the
$C^{k-2}$  topology are also derivatives in the sense of the $C^0$ topology.
Using the uniqueness of the derivative,
if we have formulas for  the derivative in
the $C^0$ topology, these are also formulas for the derivative
in the $C^{k - 2}$ topology.

Hence, we can conclude that the reminders in the expansions
the form $O_{C^0}(\eps^{1 + \rho})$ for any  $\rho > 0$ are
actually
$O_{C^{k -2}}( \eps^2)$.
\end{rem}

If we now denote by $\mathcal{M}( \tau, I, \vphi, t)$  the Melnikov function
which is the first order approximation of the distance between the
invariant manifolds in \eqref{melnikoveps}, given by
\begin{equation}\label{melnikov}\begin{split}
{\mathcal{M}}&( \tau, I, \vphi, t) \\&=-\frac{\partial }{\partial \tau}
\int_{-\infty}^\infty \, \left[h( p^0(\tau+\sigma\bar 1 ),
q^0(\tau+\sigma \bar 1;0), I , \vphi +  \sigma\omega(I),
t+\sigma;0)\right.
\\ &\qquad\qquad\qquad\left.- h( 0, 0, I , \vphi +  \sigma\omega(I),
t+\sigma;0)\right] d\sigma,\end{split}
\end{equation}
then \eqref{melnikoveps} yields
\begin{equation}\label {splittingsize}\Psi^\un_\eps( \tau, I, \vphi, t)
-\Psi^\st_\eps( \tau, I, \vphi, t)=-\eps \mathcal{M}(\tau,I,\vphi,
t)+O_{C^{k}}(\eps^{1+\varrho}).\end{equation}

\begin{rem}
We can alternatively write the Melnikov function \eqref{melnikov}
as
\begin{equation}\label{melnikovalternative}
\begin{split}\mathcal{M}( \tau, I, \vphi, t) &\\= \int_{-\infty}^\infty &\,\left[
[P,h]( p^0(\tau+\sigma\bar 1 ), q^0(\tau+\sigma\bar 1 ), I , \vphi
 +\sigma\omega(I), t+\sigma;0)
\right.\\&\left.- [P,h]( 0, 0,
I , \vphi +  \sigma\omega(I),t+\sigma;0)\right] d\sigma.\end{split}
\end{equation}

This is due to the fact that, by \eqref{geometrymelnikov},
\begin{eqnarray*}
\lefteqn{[P,h](
p^0(\tau+\sigma\bar 1 ), q^0(\tau+\sigma\bar 1 ), I , \vphi +
\sigma\omega(I),t+ \sigma;0)}\\&&=-\frac{\partial }{\partial \tau}
h( p^0(\tau+\sigma\bar 1 ), q^0(\tau+\sigma\bar 1 ), I , \vphi +
\sigma\omega(I) ,t+ \sigma;0),
\end{eqnarray*}
where the Poisson bracket is thought of as the derivative along
the trajectory.  Similar calculations  can be found in
\cite{DelshamsG00,DelshamsLS06b}.
\end{rem}

\subsection{Geometric properties of the Melnikov vector}
Recall that the Melnikov vector \eqref{eqn:melnikov_vect} is the first order
expansion in $\eps$ of the vertical distance of the invariant
manifolds.

A simple application of the implicit function theorem
shows that if the Melnikov vector  has a non-degenerate
zero (i.e., a zero so that the gradient has full rank),
then, there is a true intersection which, furthermore is
transversal.  Therefore, the intersection will be locally a manifold
$\Gamma_\eps$ of the same dimension as $\Lambda_0$ and it can be continued
smoothly when $| \eps| < \eps_0$ and  the intersection is
transversal when $0 < |\eps| < \eps_0$.
Note that the manifold itself has a smooth continuation
trough $\eps = 0$, even if it is not transversal there.

These intersections that continue smoothly through $\eps = 0$ are
called in the literature
the \emph{primary} intersections. There are
other transversal intersections (called \emph{secondary} )
which cannot be continued across $\eps = 0$. See \cite{Moser73}
for a more detailed discussion.

\begin{rem}\label{rem:trasninter}
It is well known to dynamicists that, if there is one transverse
intersection, because of the invariance of the manifolds,
the orbit of this intersection are  also a transverse
intersection.

As we will see later, one  expects that for
generic systems with generic perturbations there are several  primary
intersections which are not in the same orbit.  This is
because one expects, because of topological reasons
that there are several non-degenerate zeros of the Melnikov function
in a fundamental domain.

Each of these primary intersections could be used for problems in
Arnol'd diffusion since each of them generates a scattering map.
It also seems possible that, using
genericity arguments, etc.  one could use the secondary intersections,
but, since they are not perturbative, they are not accessible
to the perturbative computations considered in this paper.
\end{rem}

In our case, the existence of non-degenerate zeros of the
Melnikov vector  is the content of
assumption~\textbf{H3}.

For the applications to concrete systems, note that the integrals
can be computed numerically with error bounds and, using
an implicit function theorem, index theorem etc. one can
verify \textbf{H3} with a finite precision calculation.
Indeed the CAPD software \cite{CAPD} provides software
to perform similar  calculations in a computer with rigorous bounds.
See also \cite{Tucker11} for another general use package.
The paper \cite{CapinskiGL17} carries out similar computations
for concrete problems in celestial mechanics.

Note that
\eqref{melnikov} can be written as
\begin{equation} \label{melnikov2}
\mathcal{M}( \tau, I, \vphi,t) = \frac {\partial} {\partial{\tau} }
{\tM}( \tau, I, \vphi,t)
\end{equation}
where $\tM$ is the Melnikov potential
defined by \eqref{eqn:melnikov_int}.

When the potential $\mathcal{M}$ has a non-degenerate critical
point, by the implicit function theorem, we can find a
$C^{k-1}$-differentiable $(2d+1)$-dimensional manifold of critical points
given by
\begin{equation}\label{intersection}
\tau = \tau^*(I, \vphi,t)
\end{equation}
for all $(I, \vphi,t) \in U^-_\eps$, where $ U^-_\eps$ is an
open domain in $\tilde\Lambda_\eps$. We denote by $\tilde\Gamma _\eps$ the  intersection manifold
parametrized by $\tau = \tau^*(I, \vphi,t)$. The domain $U^-_\eps$
can be chosen so that it contains a ball of positive size
independent of $\eps$; per assumption \textbf{H3}, we can choose $U^-_\eps$ to contain the image $\tilde\Xi_\eps(\mathcal{I}^*\times\mathcal{O})$ of the product $\mathcal{I}^*\times\mathcal{O}$ via $\Xi_\eps$.
Here $\Xi_\eps$ is the parametrization introduced in Remark \ref{parameterization1}.
By restricting $U^-_\eps$ if necessary, we
can ensure that $\tilde{\Omega}^\un_\eps$ is a diffeomorphism from
$U^-_\eps$ to the intersection surface.

If the critical point is non-degenerate, the intersection between
stable and unstable manifolds will be transverse along a manifold
parameterized as in~\eqref{intersection}.

It is important to note that, under hypothesis \textbf{H3},
all the primary intersections we consider are given as
graphs from $(I,\vphi, t)$, the variables in $\Lambda_0$, to
$\tau$. Geometrically, this means that the intersections
produced using hypothesis \textbf{H3} are transverse, and they also
satisfy the strong transversality conditions \eqref{goodtransversal} and \eqref{goodtransversal2}, that ensure that we can define  the scattering
map as a differentiable map. The reason why we have \eqref{transveral2}
is that we have it for all graphs with finite derivative for $\eps = 0$.

This concludes the proof of Proposition \ref{prop:melnikov_int}.

\begin{rem}
The assumption \textbf{H3} is more than what we need for
subsequent analysis. The subsequent analysis
just needs \eqref{goodtransversal} and we are verifying
\eqref{goodtransversal2}. It is not hard to produce
examples -- that happen in a $C^1$ open set of
diffeomorphisms -- where \eqref{goodtransversal2} fails
but nevertheless \eqref{goodtransversal} holds.
\end{rem}

\subsection{Reduced Melnikov potential}\label{section:reduced}

Note that by the change of variable formula applied to \eqref{eqn:melnikov_int} we have that
\[\widetilde{\mathcal{M}}=\widetilde{\mathcal{M}}(\tau-t\bar 1 ,I,\vphi-t\omega(I),0)\]
so the map \[\tau \in\mathbb{R}^n\mapsto \widetilde{\mathcal{M}}(\tau,I,\vphi-t\omega(I),0)\in\mathbb{R}\]
has critical points $\tau^*$ satisfying
\begin{equation}\label{eqn:tau_shift}\tau^*(I,\vphi-t\omega(I),0)=\tau^*(I,\vphi,t)-t\bar 1.\end{equation}
Let \[ \mathcal{M}^*(I,\vphi,t)=\widetilde{\mathcal{M}}(\tau^*(I,\vphi,t),I,\vphi,t).\]
Note that the time-shift property \eqref{eqn:tau_shift} yields the following invariance equation for $\mathcal{M}^*$
\begin{equation}\label{eqn:m_shift}
\mathcal{M}^*(I,\vphi,t)=\mathcal{M}^*(I,\vphi-t\omega(I),0).
\end{equation}
Then, the  mapping \begin{equation}\label{eqn:m_reduced} {\widetilde{\mathcal{M}}^*}(I,\theta)={\mathcal{M}}^*(I,\theta,0),\end{equation}
is  well defined for all $ (I,\theta)$ with $I\in\mathcal{I}^*$  and $\theta\in\mathbb{R}^d$ of the form $\theta=\vphi-t\omega(I)$, that is, for all $\theta\in\mathbb{R}^d$ for which $(\theta+t\omega(I),t) \in \mathcal{O}$ for some $t$.

Conversely, if $(I,\vphi,t)\in \mathcal{I}^*\times\mathcal{O}$, then $ (I,\vphi-t\omega(I))$ is well defined, and we have\
\begin{equation}\label{eqn:m_equiv} {\widetilde{\mathcal{M}}^*}(I,\vphi-t\omega(I))=\mathcal{M}^*(I,\vphi-t\omega(I),0)=\mathcal{M}^*(I,\vphi,t).
\end{equation}

\subsection{Proof of Theorem \ref{prop:melnikov_jump}}\label{subsection:scattering}
Let ${\tilde {x}}$ be a point in the homoclinic intersection $\tilde\Gamma_\eps$ specified above,
let ${\tilde {x}}^-=(\tilde\Omega^\st_\eps)^{-1}({{\tilde x}})$ and ${\tilde {x}}^+=\tilde\Omega
^\un_\eps({\tilde {x}})=\tilde S_\eps({\tilde {x}}^-)$.

Applying the fundamental theorem of calculus:
\begin{equation}\label{firstorderscattering1}
\begin{split}
I({{\tilde x}}^+) - I({{\tilde x}}) & =
\int_0^\infty \frac{d}{d \sigma} \left[I( \tilde{\phi}^\sigma_\eps({\tilde {x}}))
-I( \tilde{\phi}^\sigma_\eps(\tilde{\Omega}^\st_\eps({\tilde {x}})))\right]\, d\sigma\\
&= \eps \int_0^\infty (\X^1 I)( \tilde{\phi}^\sigma_0({\tilde {x}})
)-(\X^1 I)(\tilde{\phi}^\sigma_0(\tilde{\Omega}^\st_0({\tilde {x}})))\, d\sigma\\& +
O_{C^{k}}(\eps^{1+\varrho}) \\
&= \eps \int_0^\infty [ I, h]( \tilde{\phi}^\sigma_0({\tilde {x}})
)-[I,h](\tilde{\phi}^\sigma_0(\tilde{\Omega}^\st_0({\tilde {x}})))\, d\sigma\\& +
O_{C^{k}}(\eps^{1+\varrho})
\end{split}
\end{equation}
for any $0<\varrho < 1$.

The application of the fundamental theorem of calculus
to improper integrals requires some justification
and the fact that we substitute the
$\phi^\sigma_\eps$ by $\phi^\sigma_0$ in the first line
requires some justification, but it
is very similar to the argument we used before. Since
the integrand converges exponentially
fast to zero as $\sigma\to\infty$, we can take
the limits to be $C |\log( |\eps| )| $ rather than  $\infty$
incurring in an error bounded by $C \eps^2$.
Then, we change the dynamics by the unperturbed one in
this finite interval incurring in an error bounded by
$C \eps^s|\log( |\eps| )$|. Finally, we change the limit of
integration from $C |\log( |\eps| )|$ to $\infty$.

The computation for the reversed dynamics gives us that
\begin{equation}\label{firstorderscattering2}
\begin{split}
I({{\tilde x}}^-) - I({{\tilde x}})
= &-\eps \int_{-\infty}^{0}  (\X^1 I)( \tilde{\phi}^\sigma_0({{\tilde x}})
)-(\X^1 I)(\tilde{\phi}^\sigma_0(\tilde{\Omega}^\st_0({{\tilde x}})))\, d\sigma\\& +
O_{C^{k}}(\eps^{1+\varrho})\\
= &-\eps \int_{-\infty}^{0}  [I, h]( \tilde{\phi}^\sigma_0({{\tilde x}})
)-[I,h](\tilde{\phi}^\sigma_0(\tilde{\Omega}^\st_0({{\tilde x}})))\, d\sigma\\& +
O_{C^{k}}(\eps^{1+\varrho}).
\end{split}
\end{equation}

As in the theory of the scattering map (See Appendix \ref{sec:scatteringnhim})
it natural to compute the change of action between the
asymptotic orbit in the future and the asymptotic orbit in the past.
Running the same argument backwards, we obtain in the Hamiltonian case.
\begin{equation}\label{firstorderscattering}
\begin{split}
I({{\tilde x}}^+) - I({{\tilde x}}^-) &=
\eps [I,\widetilde{\mathcal{M}}](
\tau^*(I, \vphi,t), I,
\vphi,t)+ O_{C^{k}}(\eps^{1+\varrho}),
\end{split}
\end{equation}
where $\tilde{\mathcal{M}}$ is the same Melnikov potential which appeared
in the computations of the transverse
intersection, and $\tau^*$ is the parametrization of the chosen
intersection. The convergence of the above integral follows from
simple estimates as in \cite{DelshamsLS03}.

Since \[[I,\widetilde{\mathcal{M}}]=\frac{\partial I}{\partial
\vphi}\frac{\partial \widetilde{\mathcal{M}}}{\partial I}-\frac{\partial
I}{\partial I}\frac{\partial\widetilde{\mathcal{M}}}{\partial
\vphi}=-\frac{\partial \widetilde{\mathcal{M}}}{\partial \vphi},\] and, since by \eqref{eqn:m_equiv} we have
\[
\widetilde{\mathcal{M}}(\tau^*(I,\vphi,t),I,\vphi,t)= \widetilde{\mathcal{M}}^*(I,\vphi-t\omega(I)),\]
we obtain  the desired conclusion
 that the hypothesis \textbf{H4}
implies the estimate from the statement of Theorem \ref{prop:melnikov_jump}.

A different perspective on the formula \eqref{firstorderscattering}
presented from the point of
view of the scattering map is in Appendix~\ref{sec:scatteringnhim}.
The theory of the scattering map actually gives some more
information since it allows also to compute the effects of homoclinic
excursions in
the slow variables. On the other hand, the theory presented here
relies exclusively on the fact that some of the coordinates are
slow.

\subsection{Proof of Proposition  \ref{prop:genericity} }\label{section:genericity}

The standing assumption for this section is that the perturbation is Hamiltonian  as in  \eqref{eqn:h}, and periodic, i.e., $t\in\torus$.

The facts that the conditions \textbf{H3} and \textbf{H4} are
$C^{k}$-open follow from standard stability arguments. For the
genericity part, the only non-trivial issue is to argue that the
Melnikov potential $\tilde{\mathcal{M}}$ always has critical points.
In such a case, a $C^{\infty}$ small perturbation
can make them transversal.

In the case when $n=1$, existence of
critical points  is discussed in \cite{DelshamsLS03}
(we will go over the argument at the end of this section).
For the higher dimensional case, we will argue that, in some
regime (the widely separated $tau_i$) we can consider
the problem as a sum of independent one dimensional cases.
Physically, this is clear because the jumps executed by a pendulum
are localized in time. So that, the effects of two jumps at widely
different times is independent.
Thus, showing the existence
of critical point for $\widetilde{\mathcal{M}}$ is reduced to  the
$1$-dimensional case. The argument presented here
is very conservative and, besides the critical points produced
by the above procedure, we expect that there are many more.

In the formal statement,  of the genericity results, we are including
the assumption that the perturbation is periodic. The main reason
is to be able to quote \cite{DelshamsLS03} for the genericity arguments
in the one-dimensional case. The part of the argument
that shows that the total Melnikov function is the sum of the individual
ones works without any change for more general time dependence of
the perturbations.

\subsubsection{Preliminary lemmas}

Define a Melnikov potential $\widetilde{\mathcal{M}}_i:\mathbb{R}\times
\mathcal{I}^*\times\mathbb{T}^d\times\mathbb{T}\to \mathbb{R}$
corresponding to the $i$ pendulum by
\begin{equation*}\begin{split}
\widetilde{\mathcal{M}}_i&(\varsigma,I,\vphi,t) \\ &=-\int_{-\infty}^{\infty}\left
[h(0, \ldots, p^0_i(\varsigma+\sigma),q^0_i(\varsigma +\sigma),
\ldots,0, I,\vphi+\sigma\omega(I),t+\sigma;0)\right .\\ &\qquad\qquad\left
.-h(0,\ldots,0, I,\vphi+\sigma\omega(I),t+\sigma;0)\right ]d\sigma.\end{split}
\end{equation*}

This Melnikov potential $\widetilde{\mathcal{M}}_i$ is a tool to measure
the splitting of the stable and
unstable manifolds for the equilibrium point of the $i$-th
pendulum when the other penduli have been set to rest.
As we have seen, the gradient of this function gives
the leading term in $\eps$ of the difference of the functions
whose graphs are the stable and unstable manifolds of
the stable manifold.

By making $\mathcal{I}^*$ slightly smaller   we can assume
that it is a closed ball and $\tau^*(I,\vphi,t)$ is defined for all $I\in\mathcal{I}^*$ and $(\vphi,t)\in\mathcal{O}$.

As indicated above,
It will be easy to show that the $\widetilde{\mathcal{M}}_i$
have non-degenerate zeros. See Lemma~\ref{lem:gen1}.
  Then,  we will argue that if all the  $\tau_i$
are very different from each other, the total Melnikov function is
close (in a smooth sense) to the sum of the
$\widetilde{\mathcal{M}}_i$. See  Lemma~\ref{lem:gen2}.
  Hence, we can get the existence of
non-degenerate zeros.

We expect that the above argument could work
for more general forms of the time dependence. The fact
that for widely  separated $\tau_i$, the Melnikov function is approximately
the sum of
of the partial  Melnikov functions is very clear. The fact that the
partial Melnikov functions have non-degenerate critical points is
very plausible for more general dependence, but it is very
cumbersome to write.

 We also believe that, besides the non-degenerate critical
points produced in the present argument, there are others.

\begin{lem}\label{lem:gen1}
If  for each $I\in \mathcal{I}^* $, the map $(\varsigma,\vphi,\sigma)\to
\widetilde{\mathcal{M}}_i(\varsigma,I,\vphi,\sigma)$ is not
constant on any compact disk in $\mathbb{T}^d\times\mathbb{T}$, then $\widetilde{\mathcal{M}}_i$ has at least one critical point
with respect to~$\varsigma$.
\end{lem}

\begin{proof}
The hypothesis means that, for each fixed $I$ and each compact disk $B\subseteq \mathbb{T}^d\times\mathbb{T}$,
the mapping $(\varsigma,\vphi,t)\in \mathbb{R}\times B\mapsto \widetilde{\mathcal{M}}_i(\varsigma,I,\vphi,\sigma)\in\mathbb{R}$ is not identically constant.

By performing  a change of variable $\sigma=\sigma+\sigma'$ we first notice that
\[\widetilde{\mathcal{M}}_i(\varsigma,I,\vphi,t)=\widetilde{\mathcal{M}}_i(\varsigma+\sigma',I,\vphi+\sigma'\omega(I),t+\sigma'),\]
for any $\sigma'\in \mathbb{R}$. If we choose $\sigma'=-\varsigma$, we obtain
\begin{equation}\label{melnikovtranslation1}
\widetilde{\mathcal{M}}_i(\varsigma,I,\vphi,t)
=\widetilde{\mathcal{M}}_i(0,I,\vphi-\varsigma\omega(I),t-\varsigma).\end{equation}

 Fix $I\in\mathcal{I}^*$.

Assume first that each frequency $\omega_{j}(I)$, $j=1,\ldots, d$, is a rational number. Then there exists $\nu\in\mathbb{Z}\setminus\{0\}$, such that each $\nu\omega_j(I)$ is a vector of integer components.
Then \eqref{melnikovtranslation1} for $\varsigma=\kappa\nu$,
$\kappa\in\mathbb{Z}$, leads to
\[
\widetilde{\mathcal{M}}_i(\kappa\nu,I,\vphi,t)=\widetilde{\mathcal{M}}_i(0,I,\vphi-\kappa\nu\omega(I),t-\kappa\nu)=
\widetilde{\mathcal{M}}_i(0,I,\vphi,t),
\]
since the variables $\vphi$ and $\sigma$ are defined (mod $1$).
This means that $\widetilde{\mathcal{M}}_i$ is periodic in $\varsigma$ with a
period of $\nu$. It follows that the restriction of $\widetilde{\mathcal{M}}_i$
to each of the compact domains $[\kappa\nu,(\kappa+1)\nu]\times \mathcal{I}^*\times
\mathbb{T}^d\times \mathbb{T}$, with $\kappa\in\mathbb{Z}$, has a
critical point (see Remark \ref{rem:trasninter}).

We now assume that at least one $\omega_j(I)$ is an irrational number.
Then the linear flow
$\varsigma\in\mathbb{R}\to(\vphi-\varsigma\omega(I),t-\varsigma)\in\mathbb{T}^d\times\mathbb{T}$ fills up densely
some $d^\prime$-dimensional torus $\mathcal{T}_{I}$ in $\mathbb{T}^d\times\mathbb{T}$, with $d^\prime\leq d+1$.
(In the case when $\omega(I)$ is non-resonant, that is $\bar\kappa\cdot\omega(I)+\kappa_{d+1}\neq 0$ for all $\bar\kappa\in\mathbb{Z}^d\setminus\{0\}$, $\kappa_{d+1}\in\mathbb{Z}$,  the linear flow fills up the whole
$\mathbb{T}^d\times\mathbb{T}$.)
Choose a $(d+1)$-dimensional ball $\mathcal{B}\subseteq \mathbb{T}^d\times\mathbb{T}$
with $\mathcal{B}\cap\mathcal{T}_I\neq\emptyset$ on which the map $(\varsigma,\vphi,\sigma)\to
\widetilde{\mathcal{M}}_i(\varsigma,I,\vphi,\sigma)$ is non-constant. Then, there exist $\upsilon_0\in\mathbb{R}$,
$\upsilon>0$ and some open sets $\mathcal{U}^-$, $\mathcal{U}^+$
in $\mathcal{B}$ such that
$\widetilde{\mathcal{M}}_i(\{0\}\times\mathcal{U}^-)\subseteq (-\infty, \upsilon_0-\upsilon)$, and
$\widetilde{\mathcal{M}}_i(\{0\}\times\mathcal{U}^+)\subseteq (\upsilon_0+\upsilon, +\infty)$.
Choose and fix $(\vphi, \sigma)\in\mathcal{B}$. Since the flow
$\varsigma\in\mathbb{R}\to(\vphi-\varsigma\omega(I), t
-\varsigma)\in\mathbb{T}^d\times\mathbb{T}$ fills up densely $\mathcal{T}_I$, it
follows that the trajectory of the flow will visit each of the
open sets $\mathcal{U}^-,  \mathcal{U}^+$ infinitely many times.
More precisely, there exits a sequence of times
\[\ldots<\varsigma_{j-1}^+<\varsigma_j^-<\varsigma_j^+<\varsigma_{j+1}^-<\ldots \]
with $\varsigma_j^{-,+}\to \pm\infty$ as $j\to\pm\infty$, such
that, for each $j\in\mathbb{Z}$ we have
\[\begin{split}
(\vphi-\varsigma_j^-\omega(I), t-\varsigma_j^-)&\in
\mathcal{U}^-,\\(\vphi-\varsigma_j^+\omega(I), t-\varsigma_j^+)&\in
\mathcal{U}^+.
\end{split}\]

Therefore, there exists a sequence of intervals
$J^i_j=[\varsigma^-_j,\varsigma^-_{j+1}]$, moving from $-\infty$
to $+\infty$, such that
\[\widetilde{\mathcal{M}}_i(\varsigma,I,\vphi,t)_{\mid {\varsigma\in\partial
J_j^i}}<\sup\{\widetilde{\mathcal{M}}_i(\varsigma,I,\vphi,t)\,|\,\varsigma\in
J_j^i \}-2\upsilon,
\]
since  values of $\widetilde{\mathcal{M}}_i$ on the boundary are less than
$\upsilon_0-\upsilon$ and the supremum of $\widetilde{\mathcal{M}}_i$ on the interval is more
than $\upsilon_0+\upsilon$. This implies that $\widetilde{\mathcal{M}}_i$ has a critical point
(relative to the variable $\varsigma$) in each interval $J^i_j$ (see
Remark \ref{rem:trasninter}). The critical point  yields the
$\varsigma$ coordinate  as an implicit function of $(I,\vphi, \sigma)$.

We want to emphasize that, in the above argument, the number $\upsilon>0$
is independent of the choice of the intervals.
\end{proof}

\begin{lem}\label{lem:gen2}
For every $\upsilon >0$ and every $(I,\vphi,t)$ there exists
$\Delta T>0$ such that for every $(\tau_1,\ldots,\tau_n)$ with
$0<\Delta T<\tau_{i+1}-\tau_i$ for all $i=0,\ldots, n-1$ we have
\begin{eqnarray*}\left|\widetilde{\mathcal{M}}(\tau_1,\ldots,
\tau_n,I,\vphi,t)- \left(\widetilde{\mathcal{M}}_1(\tau_1,I,\vphi,t)+\ldots
+\widetilde{\mathcal{M}}_n(\tau_n,I,\vphi,t\right))\right|<\upsilon.
\end{eqnarray*}
\end{lem}
\begin{proof} Let $(I,\vphi,t;\eps)$ be fixed.
Let $K$ the maximum of the norm of the Jacobi matrix of $h$ with
respect to $(p,q)$ on a path connected  and compact neighborhood
$\mathcal{N}$ of the homoclinic manifold $\Sigma^*$. There exists
$T_1,T_2\in\mathbb{R}$ such that
\[\begin{split}&C\frac{\max\{\lambda^+,1/\mu^-\}^{T_1}}{\ln\left(\max\{\lambda^+,1/\mu^-\}\right) }
<\frac{\upsilon}{4n(n+2)K},
 \textrm { and}\\ &C\frac{\max\{\mu^+,1/\lambda^-\}^{-T_2}}{\ln\left(\max\{\mu^+,1/\lambda^-\}\right) }
<\frac{\upsilon}{4n(n+2)K}.\end{split}\] Let $\Delta T=T_2-T_1$,
and let $(\tau_1,\ldots,\tau_n)$ be as above. Then the  real axis
can be split into $(n+2)$ intervals such that on the first and
last intervals all penduli are close to rest, and on each of the
intermediate $n$ intervals there is exactly one pendulum that
makes wide swings, while the rest of the penduli are close to
rest. These intervals are given by
\[
\begin{split}
&L_0=(-\infty,T_1-\tau_n],\, L_1=[T_1-\tau_n,
T_1-\tau_{n-1}],\,
\ldots\\
& \ldots, L_{k}=[T_1-\tau_{n-k+1}, T_1-\tau_{n-k}],\ldots \\
&L_{n-1}=[T_1-\tau_{2}, T_1-\tau_{1}],L_{n}=[T_1-\tau_{1},+\infty).
\end{split}
\]

When $\sigma\in L_0 $  we have that
$(p^0_i(\tau_i+\sigma),q^0_i(\tau_i+\sigma))\approx(0,0)$
for all $i=1,\ldots,n$; when $\sigma\in L_k$,
$k=2,\ldots,n$  we have that
$(p^0_i(\tau_i+\sigma),q^0_i(\tau_i+\sigma))\approx (0,0)$ for all
$i=1,\ldots,n$ with $i\neq k$; when $\sigma\in L_n$ we have that
$(p^0_i(\tau_i+\sigma),q^0_i(\tau_i+\sigma))\approx (0,0)$  for
all $i=2,\ldots,n$; when  $\sigma\in L_{n+1} $  we have that
$(p^0_i(\tau_i+\sigma),q^0_i(\tau_i+\sigma))\approx (0,0)$ for all
$i=1,\ldots,n$.  More precisely, we claim that the following
inequalities hold true
\begin{equation}\label{claim1}
\int_{L_0}|h(p^0(\tau+\sigma),q^0(\tau+\sigma),*)-h(0,\ldots,0,I,*)|d\sigma<\frac{\upsilon}{2(n+2)},
\end{equation}
\begin{equation}\label{claim2}\begin{split}
&\int_{L_k}|h((p^0(\tau+\sigma),q^0(\tau+\sigma),*)\\&\qquad-h(0,\ldots,0,p^0_k(\tau_k+\sigma),0,
q^0_k(\tau_k+\sigma),\ldots,*)|d\sigma\\
&\quad<\frac{\upsilon}{2(n+2)}, \textrm { for all } k=1,\ldots,n,
\end{split}
\end{equation}
\begin{equation}\label{claim3} \int_{L_{n+1}} |h(p^0(\tau+\sigma),q^0(\tau+\sigma),*)
-h(0,\ldots,0,*)|d\sigma<\frac{\upsilon}{2(n+2)},\end{equation}
\begin{equation}\label{claim4}\begin{split}\int_{\mathbb{R}\setminus
L_k}|h(0,\ldots,p^0_k(\tau_k+\sigma),q^0_k(\tau_k+\sigma),\ldots,0,*)-h(0,\ldots,0,*)|d\sigma
\\<\frac{\upsilon}{2(n+2)} \textrm { for all } k=1,\ldots,n,\end{split}
\end{equation}
where $*=(I,\vphi+\sigma\omega (I), t +\sigma;0)$.

We are going to argue, for example \eqref{claim2}; the other
inequalities will follow similarly. Applying the mean value theorem
along smooth curves $\gamma$ in $\mathcal{N}$ with $\gamma(0)={\tilde x}$ and
$\gamma (1)={\tilde y}$, we obtain
\[\|f({\tilde x})-f({\tilde y})\|\leq\|f'\|_{C^0}\inf_{\gamma}\int _0^1 |\gamma'(t)|dt\leq C\|{\tilde x}-{\tilde y}\|\]
for some $C>0$, for all ${\tilde x},{\tilde y}\in\mathcal{N}$. Applying this
inequality to the mapping $(p,q)\to h(p,q,*)$ for the pair of points
${\tilde x}=(p^0(\tau+\sigma),q^0(\tau+\sigma))$,
${\tilde y}=(0,\ldots,p^0_k(\tau_k+\sigma),q^0_k(\tau_k+\sigma),\ldots,0)$,
we obtain
\[\begin{split}|h((p^0(\tau+\sigma),q^0(\tau+\sigma),*)-
h(0,\ldots,p^0_k(\tau_k+\sigma),q^0_k(\tau_k+\sigma),\ldots,0,*)|\\
\leq K\sum_{i\neq
k}\|\left(p^0_i(\tau_k+\sigma),q^0_i(\tau_k+\sigma)\right)\|.\end{split}\]
When $\sigma\in L_k$ we have that
\[\|\left(p^0_i(\tau_i+\sigma),q^0_i(\tau_i+\sigma)\right)\|\leq C\frac{\max\{\lambda^+,1/\mu^-\}^{T_1}}{\ln\left(\max\{\lambda^+,1/\mu^-\}\right) }
\] if $i<k$ and \[\|\left(p^0_i(\tau_i+\sigma),q^0_i(\tau_i+\sigma)\right)\|\leq C\frac{\max\{\mu^+,1/\lambda^-\}^{-T_2}}{\ln\left(\max\{\mu^+,1/\lambda^-\}\right) }
\] if $i>k$. This implies that
\[\int_{L_k}K\|\left(p^0_i(\tau_k+\sigma),q^0_i(\tau_k+\sigma)\right)\|d\sigma\leq
\frac{\upsilon}{4n(n+2)}\] for each $i\neq k$. Adding these
inequalities for all $i\neq k$ yields \eqref{claim2}.

Using the inequalities \eqref{claim1}, \eqref{claim2},
\eqref{claim3} and \eqref{claim4},  we estimate
\begin{eqnarray*}
\lefteqn{\left|\widetilde{\mathcal{M}}(\tau_1,\ldots, \tau_n,I,\vphi,t)-
\left(\widetilde{\mathcal{M}}_1(\tau_1,I,\vphi,t)+\ldots
+\widetilde{\mathcal{M}}_n(\tau_n,I,\vphi,t\right))\right|}\\
&& =\left |\sum_{k=0}^{n+1} \int_{L_k} [h(p^0(\tau+\sigma),
q^0(\tau+\sigma),*)-h(0,\ldots ,0,*)]d\sigma\right.
\\&&\phantom{AAA} \left.-\sum_{k=1}^{n}
\int_{-\infty}^{\infty}[h(0,\ldots,p^0_k(\tau_k+\sigma),q^0_k(\tau_k+\sigma),\ldots,0,*)
-h(0,\ldots ,0,*)]d\sigma\right|\\&& \le
\frac{2\upsilon}{2(n+2)}+\left |\sum_{k=1}^{n} \int_{L_k}
[h(p^0(\tau+\sigma),
q^0(\tau+\sigma),*)-h(0,0,*)]d\sigma\right.\\&&
\phantom{AAA} \left.-\sum_{k=1}^{n}
\int_{-\infty}^{\infty}[h(0,\ldots,p^0_k(\tau_k+\sigma),q^0_k(\tau_k+\sigma),\ldots,0,*)
-h(0,\ldots ,0,*)]d\sigma\right|\\&& \le
\frac{2\upsilon}{2(n+2)}+\left |\sum_{k=1}^{n} \int_{L_k}
[h(p^0(\tau+\sigma),
q^0(\tau+\sigma),*)-h(0,0,*)]d\sigma\right.\\&&
\phantom{AAA} \left.-\sum_{k=1}^{n}
\int_{L_k}[h(0,\ldots,p^0_k(\tau_k+\sigma),q^0_k(\tau_k+\sigma),\ldots,0,*)
-h(0,\ldots ,0,*)]d\sigma\right|\\&&
\phantom{AAA} +\left |\sum_{k=1}^{n}
\int_{\mathbb{R}\setminus
L_k}[h(0,\ldots,p^0_k(\tau_k+\sigma),q^0_k(\tau_k+\sigma),\ldots,0,*)
-h(0,\ldots ,0,*)]d\sigma\right|\\&& \le
\frac{2\upsilon}{2(n+2)}+\sum_{k=1}^{n} \int_{L_k} \left
|h(p^0(\tau+\sigma),
q^0(\tau+\sigma),*)\right.\\&&\qquad\qquad\qquad\quad\left.
-h(0,\ldots,p^0_k(\tau_k+\sigma),q^0_k(\tau_k+\sigma),\ldots,0,*)
\right|d\sigma+\frac{n\upsilon}{2(n+1)}\\&&\le
\frac{2\upsilon}{2(n+2)}+\frac{n\upsilon}{2(n+1)}+\frac{n\upsilon}{2(n+1)}=\upsilon.
\end{eqnarray*}
This concludes the proof of the lemma.
\end{proof}

\begin{lem}\label{lem:gen3}
If  for each $I\in\mathcal{I}^*$ the map $(\tau,\vphi,t)\to
\widetilde{\mathcal{M}}(\tau,I,\vphi,t)$ is not constant over any compact disk in $\mathbb{T}^d\times\mathbb{T}$, then
$\widetilde{\mathcal{M}}$ has at least one critical point with respect to
$\tau$.
\end{lem}
\begin{proof}
The case when all frequencies $\omega_j(I)$ are rational follows as in the proof of
Lemma \ref{lem:gen1}.

We now assume that at least one $\omega_j(I)$ is an irrational number. As before,  the linear flow
$\tau\in\mathbb{R}\to(\vphi-\tau\omega(I),t-\tau)\in\mathbb{T}^d\times\mathbb{T}$ fills up densely
some $d^\prime$-dimensional torus $\mathcal{T}_{I}$ in $\mathbb{T}^d\times\mathbb{T}$, with $d^\prime\leq d+1$.
Choose a $(d+1)$-dimensional ball $\mathcal{B}\subseteq \mathbb{T}^d\times\mathbb{T}$
with $\mathcal{B}\cap\mathcal{T}_I\neq\emptyset$ on which the map $(\tau,\vphi,\sigma)\to
\widetilde{\mathcal{M}}_i(\tau,I,\vphi,\sigma)$ is non-constant.

Since
$\widetilde{\mathcal{M}}$ is non-constant in $\tau$ over any compact disk in $\mathbb{T}^d\times\mathbb{T}$, there exist
$\upsilon_0\in\mathbb{R}$, $\upsilon>0$ and some open sets $\mathcal{U}^-$,
$\mathcal{U}^+$ in the space of coordinates $(\vphi,\sigma)$ such
that $\widetilde{\mathcal{M}}(\mathcal{U}^-)\subseteq (-\infty, \upsilon_0-\upsilon)$,
and $\widetilde{\mathcal{M}}(\mathcal{U}^+)\subseteq (\upsilon_0+\upsilon, +\infty)$.
Choose and fix $(\vphi, \sigma)$.

For each $i$, there exists a sequence of intervals $J^i_j$, moving
from $-\infty$ to $+\infty$, such that
\[\widetilde{\mathcal{M}}_i(\tau_i,I,\vphi,t)|_{\tau_i\in\partial
J_j^i}<\sup\{\widetilde{\mathcal{M}}_i(\tau_i,I,\vphi,t)\,|\,\tau_i\in J_j^i
\}-2\upsilon.
\]
Let $\Delta T$ be a number as in Lemma \ref{lem:gen2}
corresponding to the value of $\upsilon/2$. From the above
sequence
$\{J^i_j\}_{{i=1,\ldots,n}\atop{j\in\mathbb{Z}\quad\,\,\,\,\,}}$
we can select a collection of intervals
\[J^1_{j_1}, \ldots,J^n_{j_n}\] that are $\Delta T$ apart one from
the other. On the one hand we have  that
\begin{eqnarray*}  &\sum_{i=1}^n&\widetilde{\mathcal{M}}_i(\tau_i,I,\vphi,\sigma)|_{\tau\in\partial(
J_{j_1}^1\times\ldots\times
J^n_{j_n})}\\&&<\sup\{\sum_{i=1}^n  \widetilde{\mathcal{M}}_i(\tau_i,I,\vphi,\sigma)\,|\,\tau\in
J_{j_1}^1\times\ldots\times J^n_{j_n} \}-2\upsilon.
\end{eqnarray*}
On the other hand, by Lemma \ref{lem:gen2}, we have \[
\left| \widetilde{\mathcal{M}}(\tau_1,\ldots, \tau_n,I,\vphi,t)-
\left(\widetilde{\mathcal{M}}_1(\tau_1,I,\vphi,t)+\ldots
+\widetilde{\mathcal{M}}_n(\tau_n,I,\vphi,t\right))\right|<\upsilon/2,
\] for any $\tau\in J_{j_1}^1\times\ldots\times J^n_{j_n}$.
It results that
\[\widetilde{\mathcal{M}}(\tau,I,\vphi,t)|_{\tau\in\partial(
J_{j_1}^1\times\ldots\times
J^n_{j_n})}<\sup\{\widetilde{\mathcal{M}}(\tau,I,\vphi,t)\,|\,\tau\in
J_{j_1}^1\times\ldots\times J^n_{j_n} \}-\upsilon.
p\]
We conclude that there exists a critical point for
$\tau\to\widetilde{\mathcal{M}}(\tau,I,\vphi,\sigma)$ in each rectangle
$J_{j_1}^1\times\ldots\times J^n_{j_n}$ with the property that the
intervals $J^1_{j_1}, \ldots,J^n_{j_n}$  are $\Delta T$ apart one
from the other.  Each such critical point yields its
$\tau$-coordinate as an implicit function of $(I,\vphi, t)$.

Since there are infinitely many such rectangles, there are
infinitely many critical points (see Remark \ref{rem:trasninter}).
\end{proof}

\begin{rem}
Note that, since we are only arguing the existence of
local maxima is enough to perform $C^0$ estimates and
argue that the value in the boundary of a set is
smaller than the value on a point in the interior.
If we had used $C^1$ estimates, we could have shown
the existence of many other critical points.
It seems that characterizing better the multiplicity
of intersections is an interesting project, but we will
not consider it here.
\end{rem}

\subsubsection{Genericity of condition \textbf{H3} and \textbf{H4}}
\label{sec:genericity}

It is easy to see that the conditions \textbf{H3} and \textbf{H4}
are $C^{k+1}$-open, since a non-degenerate critical point for
$\tau\to\tilde{\mathcal{M}}(\tau,I,\vphi,t)$ remains non-degenerate after a
small enough $C^2$-perturbation, and if the function
$(I,\vphi,t)\to
\frac{\partial{\tilde{\mathcal{M}}}}{\partial\vphi}(\tau^*(I,\vphi,t),
I,\vphi,t)$ is non-constant over a disk in $\mathbb{T}^d\times\mathbb{T}$, then it remains non-constant after a small enough
$C^{k+1}$-perturbation.

By Lemma \ref{lem:gen3} there exists at least a critical point for
$\tau\to\tilde{\mathcal{M}}(\tau,I,\vphi,t)$ provided $\tilde{\mathcal{M}}$ is
non-constant over any disk in $\mathbb{T}^d\times\mathbb{T}$, for each $I$ fixed. This latter condition is
$C^\infty$-dense. Since a critical point as above can be made
non-degenerate by an arbitrarily  $C^\infty$-small  perturbation, it
results that the condition \textbf{H3} is $C^\infty$-dense. The
condition \textbf{H4} is also $C^\infty$-dense since $(I,\vphi,t)\to
\frac{\partial{\tilde{\mathcal{M}}}}{\partial\vphi}(\tau^*(I,\vphi,t),
I,\vphi,t)$  can be made non-constant over a disk in $\mathbb{T}^d\times\mathbb{T}$, if necessary, through an
arbitrarily small $C^\infty$-perturbation.
\qed


\section*{Acknowledgements}
We are grateful to Profs. T. M.-Seara, A. Delshams, A. Haro, Dr. M. Canadell
and Ms. J. Yang
for discussions and comments.

\appendix

\section{A  brief summary of normal hyperbolicity theory}
\label{sec:hyperbolicity}

In this section we recall standard definitions and results on the
theory of normal hyperbolicity
\cite{HirschPS77,Fenichel74,Fenichel77,Pesin04,BatesLZ08},
and on the scattering map
\cite{DelshamsLS08a}. Most of the material in this section
is standard and is included here to set the notation of the rest of
the paper and to check some of the constructions.

\subsection{Normally hyperbolic invariant manifolds}
Consider the general situation of a $C^r$-differentiable flow $\phi: M\times \mathbb{R}\to M$, $r\geq r_0$, defined on a  manifold $M$. Here $r_0$ is assumed to be suitably large.
It is important to realize that the theory does
 not assume that the manifolds are
compact. Nevertheless, we will assume that $C^r$ means that the
derivatives of order up to $r$ are continuous and uniformly bounded.
This remark already appears in \cite{HirschPS77}, and is crucial
for the infinite dimensional applications \cite{BatesLZ08}.

A submanifold   (with or without boundary) $\Lambda$ of $M$ is said
 to be a normally hyperbolic invariant manifold for $\phi^t$ if  it is invariant under $\phi^t$, and there exists a splitting of the tangent bundle
 of $TM$ into sub-bundles over $\Lambda$.
\[
T_p M=E^{\un}_p \oplus E^{\st}_p \oplus T_p \Lambda, \quad \forall p \in \Lambda.
\]
that are invariant under $D\phi^t$ for all $t\in\mathbb{R}$, and there exist  rates \[\tilde\lambda_-\le \tilde\lambda_+<\tilde\lambda_c<0<\tilde\mu_c<\tilde\mu_-\le \tilde\mu_+\]
and a constant $\tilde{C}>0$, such that for all $x\in\Lambda$ we have
\[
\begin{split} \tilde{C}e^{t\tilde\lambda_- }\|v\| \leq \|D\phi^t(x)(v)\|\leq  \tilde{C}e^{t\tilde\lambda_+ }\|v\|  \textrm{ for all } t\geq 0, &\textrm{ if and only if } v\in E^{\st}_x,\\
\tilde{C}e^{t\tilde\mu_+ }\|v\|\leq \|D\phi^t(x)(v)\|\leq  \tilde{C}e^{t\tilde\mu_- }\|v\|  \textrm{ for all } t\leq 0,  &\textrm{ if and only if }v\in E^{\un}_x,\\
\tilde{C}e^{|t|\tilde\lambda_c }\|v\|\leq  \|D\phi^t(x)(v)\|\leq  \tilde{C}e^{|t|\tilde\mu_c}\|v\| \textrm{ for all } t\in\mathbb{R}, &\textrm{ if and only if }v\in T_x\Lambda.
\end{split}
\]

The manifold $\Lambda$ is $C^{\ell}$-differentiable for
 some $\ell\leq r-1$ that depends on the rates
$\tilde\lambda^-$, $\tilde\lambda^+$, $\tilde\mu^-$, $\tilde\mu^+$,  $\tilde\lambda_c$, and $\tilde\mu_c$.
\begin{equation}\label{eqn:ratesdifferentiable}
\begin{split}
& \ell \mu_c  +  \lambda_+ < 0 \\
& \ell \lambda_c +  \mu_+ > 0
\end{split}
\end{equation}
We refer to \cite{Fenichel74} for the result and
for examples.

For $\Lambda$ normally hyperbolic,
there exist stable and unstable manifolds of $\Lambda$,
 denoted $W^{\st}(\Lambda)$ and $W^{\un}(\Lambda)$, respectively,
and are $C^{\ell-1}$-differentiable.
 They are foliated by stable and unstable manifolds of points,
$W^{\st}(x)$, $W^{\un}(x)$, respectively, which are as smooth and the flow.
These manifolds are defined by:

\begin{equation}\label{stablemanif-flow}
\begin{split}
 W^\st(\Lambda)
 =& \{ y \,|\, d(  {\phi}^t_\eps(y), \Lambda ) \rightarrow 0 \textrm{ as {$t \to +\infty$} }\} \\
 =& \{ y \,|\,   d(  {\phi}^t_\eps(y), \Lambda ) \le \tilde C_y e^{(\lambda_c -\delta) t} ,\delta > 0,  t \ge 0 \} \\
 =& \{ y \,|\, d(  {\phi}^t_\eps(y), \Lambda ) \le \tilde C_y e^{\lambda_+ t} ,t \ge 0\} \\
 W^\un(\Lambda)
 =& \{ y \,|\, d(  {\phi}^t_\eps(y), \Lambda ) \rightarrow 0 \textrm{ as {$t \to -\infty$} }\} \\
 =& \{ y \,|\,   d(  {\phi}^t_\eps(y), \Lambda ) \le \tilde C_y e^{(\mu_c + \delta) t} ,\delta > 0,  t \le 0 \} \\
 =& \{ y \,|\, d(  {\phi}^t_\eps(y), \Lambda ) \le \tilde C_y e^{\mu_+ t} ,t \ge 0\} \\
W^{\st}(x)
=& \{y | d( \phi^t (y),\phi^t (x) ) < \tilde C_y  e^{\lambda_+ t},\, t\geq 0  \} \\
=& \{y | d( \phi^t (y),\phi^t (x) ) < \tilde C_y  e^{\lambda_c -\delta) t},\,
\delta > 0, t\geq 0  \} \\
W^{\un}(x)
=& \{y | d( \phi^t (y), \phi^t (x) ) < \tilde C_y  e^{(\mu_c +\delta) t},\, \delta > 0,  t\leq0  \}.\\
=& \{y | d( \phi^t (y), \phi^t (x) ) < \tilde C_y  e^{\mu_- t},\, t\leq0  \}.
\end{split}
\end{equation}

Notice that the first characterization of $W^\st(\Lambda),W^\un(\Lambda)$
are topological in nature, but they are equivalent to
rate conditions. Indeed, the equivalence of the two
characterizations shows that there is a bootstrap of the rates.
For the manifolds $W^\st(x), W^\un(x)$, there is no topological
characterization.  There could be points in $\Lambda$ whose
orbits approach in the future  the orbit of $x$, albeit with a slower
rate than that $\lambda_+$ which characterizes the orbits of $W^\st(x)$.

In the theory of normally hyperbolic
manifolds the stable
and unstable manifolds of points are sets of
orbits which converge with a fixed exponential rate
in the future or in the past.
This should not be confused with the usage of the same name in
other areas of dynamical systems where stable and unstable manifolds
refer to points that converge (at any rate).

It is well known from normal-hyperbolicity theory \cite{Fenichel74}
that the $W^\st(x)$ provide a foliation of $W^\st(\Lambda)$.
\begin{equation}\label{decomposition}
W^{\st}(\Lambda) = \bigcup_{x \in \Lambda } W^{\st}(x),
\quad
W^{\un}(\Lambda) = \bigcup_{x \in \Lambda} W^{\un}(x),
\end{equation}
and that $W^{\st}(x) \cap W^{\st}( x')= \emptyset$,
$W^{\un}(x) \cap W^{\un}(x') = \emptyset$
if $x,x' \in \Lambda$, $x \neq x'$.

In dynamical  terms, the above remark says that
the points whose orbits converge to $\Lambda$, have orbits
that track the orbit of a map in $\Lambda$ and that the
tracking approaches with a exponential rate $\lambda_+$.   That is, we can
for all the points that converge to $\Lambda$ we can associate
a unique representative of the long term behavior.

The manifolds $W^{\st}(x)$, $W^{\un}(x)$ are
$C^r$ if the flow $\phi^t$ is $C^r$ (here, we can allow
$k = \infty, \omega$) and we have
\[
T_x W^{\st}(x) = E^{\st}_x,
\quad
T_x W^{\un}(x) = E^{\un}_x.
\]
The reason why $W^\st(\Lambda)$ may fail to be $C^r$ is
because the leaves $W^\st(x)$ -- which are $C^r$ -- may
depend  less regularly in the point $x$.

In the application of this paper, since $\lambda_c, \mu_c$ are
close to zero for $\eps$ small, we see that we can take the
regularity to be any finite number, as large as needed.


\subsection{Local stable and unstable manifolds} If we only select the set of points $y\in W^\st(\Lambda)$
with $ d({\phi}^t(y), \Lambda) \le \delta$ for all
$t\geq 0$, for some  fixed $0 < \delta \ll 1$, we obtain the
local stable manifold $W^{\st}_{\text{loc}}(\Lambda)$, which
is  a $C^k$-embedded manifold. Clearly, $W^\st(\Lambda) =
\bigcup_{t \geq 0} {\phi}^{-t} (W^{\st}_{\text{loc}}(\Lambda))$.
An analogous definition, holds, of
course, for the local unstable manifold.

Local stable and unstable manifolds are defined in a similar fashion in the  map case.

\section{The extension argument for the flows in the  models
 \eqref{generalperturbation} and \eqref{eqn:hamiltonian_eps}}
\label{sec:extension}

\subsection{Extensions of the vector field and of Hamiltonian}
In general, the theory of \cite{Fenichel71} requires to extend the perturbed
flow to a whole space. The locally invariant manifolds for the original problem
are globally invariant for the flow in the full space.

In our case, the construction of the extension is very simple.

We note that the extension of the penduli is not a big issue
since all the effects we study happen in a small perturbation
of the homoclinic intersection of the penduli. On the other
hand, the extension in the $I$ variables is crucial if
there are orbits of the perturbation when the extension changes.

Because the boundaries of the domains $\OO_1, \OO_2$ are
smooth, we can construct $C^\infty$ functions $\beta_1, \beta_2$
such that
$\beta_1:  \real^n \times \real^n \rightarrow \real$,  $\beta_1|_{\OO_1} \equiv 1$
and $\beta_1(p,q) = 0$ if $\dist( (p,q), \OO_1) > 0.1$, and
$\beta_2: \real^d \rightarrow \real$,  $\beta_2 |_{\OO_2} \equiv 1$
and $\beta_2(I) = 0$ if $\dist( I, \OO_2) > 0.1$.

In the case of Hamiltonian perturbations, we cut-off
the Hamiltonian and define the extended Hamiltonian
\begin{equation*}
\begin{split}
\tilde h_\eps(p,q, I, \vphi, t) =
&h_\eps(p,q,I, \vphi, t)(\beta_1(p,q)\beta_2(I) )\\
&+ (1 - \beta_1(p,q) \beta_2(I) )\left[ \sum_i  \pm P_i(p_i, q_i) +
\eps \omega_0 \cdot I\right].
\end{split}
\end{equation*}

Note that the perturbed hamiltonian agrees with the original one in $\D$.
What happens outside was chosen rather arbitrarily to be rather simple.
The only constraint is that the manifold $\Lambda_0$ is a NHIM for the extension.
Note that the extension depends differentiably on parameters.

In the case of general perturbations, we can consider the
extended flow to be
\begin{equation*}
\begin{split}
\tilde{\X}_\eps(p,q, I, \vphi, t) =
&\X_\eps(p,q,I, \vphi, t)(\beta_1(p,q)\beta_2(I) )\\
&+ (1 - \beta_1(p,q) \beta_2(I) )\left[ \Y^0(p,q) + \eps \Y^1(p,q,I, \vphi, t))\right],
\end{split}
\end{equation*}
where $\Y^0$, $\Y^1$ are selected to be e.g. the Hamiltonian flows corresponding to
the previous extension.

After we have an NHIM which is uniformly differentiable in a uniformly differentiable field, we
can invoke the results in the literature to obtain the persistence. As mentioned before, we
can obtain the differentiability with respect to parameters by supplementing the
evolution equations with $\dot \eps = 0$.


\subsection{Regularity with respect to parameters}

The regularity of all the objects discussed above with respect to
parameters can be easily obtained  by applying the standard argument
of considering a system extended to the phase space and the
parameter space (see \cite{Fenichel71, HirschPS77, Pesin04} for more
details on this standard argument).
The dynamics on this extended  phase space is just the
product of the dynamics we are considering and the identity dynamics
on the parameters. It is easy to see that given a normally
hyperbolic manifold, its product with the parameter space is
normally hyperbolic for the extended dynamics. Hence applying the
result of persistence for the extended system, we obtain immediately
persistence with smooth dependence on parameters in the original
system. One can also observe that the (un)stable manifolds for the
extended system are just the product of the (un)stable manifolds for
the real system and the space of parameters. Hence, the regularity
of the stable manifolds for the extended system gives the regularity
of the stable and unstable manifolds with respect to parameters of
the original system. Notice that, when considering locally invariant manifolds,
one applies this argument to a family of extended systems. If one
considers a smooth family of extended flows, one gets a smooth family of
of normally hyperbolic invariant manifolds and their stable manifolds,
which of course are locally invariant for the unextended system.

In the cases considered here, when the motion on the manifold of the unperturbed
system is integrable, the regularity can be as large as desired for small enough
$\eps$. In more general systems, the regularity is limited by an expression depending
on the ratios.

\section{The scattering map for a normally hyperbolic manifold}
\label{sec:scatteringnhim}
In this section, we cover some of
the results on the theory of the scattering map.
This allows to give an independent proof of
some of the results in the main text.

\begin{rem}
The scattering map for normally hyperbolic manifolds was
introduced in \cite{DelshamsLS00} and the theory was
developed in \cite{DelshamsLS08a}. The notation is taken from
the scattering theory in quantum mechanics.  The intuition is
that the motion in the NHIM is analogue
to the free dynamics and that the homoclinic excursions are
analogues of the coupled dynamics. The notation is
taken to resemble this. With this analogy, the
fact that the scattering map is given in first order by
the Melnikov function amounts to Fermi Golden Rule.
\end{rem}

\subsection{Wave maps}
\label{sec:wavemaps}
 One of the consequences of \eqref{decomposition} is that we can define $\Omega^\st:W^{\st}(\Lambda)\to\Lambda$
by  $x \in W^{\st}(\Omega^\st(x))$. When it cannot lead to confusion, we will often use
the abbreviation $x^+ = \Omega^\st(x)$.
Analogously, we define $\Omega^\un: W^\un(\Lambda)\to \Lambda$ by $x \in W^{u}(\Omega^\un(x))$,
and will use the abbreviation
$x^- = \Omega^\un(x)$. We will refer to  $\Omega^{\st,\un}$ as wave maps.
The meaning of $\Omega^{\st,\un}$ is that they give
points in $\Lambda$ whose orbit is exponentially
close to the orbit of $x$ as $t \to \pm \infty$.

More precisely, the  point
$\Omega^\st(x) \in \Lambda$ is characterized by
\begin{equation}\label{convergences}
d( {\phi}^t(x), {\phi}^t(\Omega^\st (x)) ) \le
\tilde C_x \tilde\lambda_+^t, \quad \textrm { for all } t \ge
0.
\end{equation}
Similarly, if $x \in W^\un(\tLambda)$, there is  a unique
point $  {\Omega}^\un(x) \in \tLambda$ such that
\begin{equation}\label{convergenceu}
d(  {\phi}^t(x), {\phi}^t(\tilde\Omega^\un (x)) ) \le
\tilde C_x \tilde\mu_-^{-t}, \quad \textrm { for all } t \leq 0.
\end{equation}

\subsection{Scattering map}
\label{sec:scatteringmap}

We summarize here the main results on the scattering map, following \cite{DelshamsLS00,Garcia00,DelshamsLS08a}. (See also \cite{ShatahZeng03}, where the scattering map, referred there to as the `return map',  is applied to Hamiltonian PDEs.)

We will present in this appendix the results for time independent perturbations.
The results for time-dependent perturbations can be obtained straightforwardly by
making the system autonomous by adding extra variables.  Note that if we want to maintain the
symplectic properties, we need to add two variables one of them standing for the time and another
one symplectically conjugate to it. See \cite{DelshamsLS08a} for more details.

Assume that $W^\un(\Lambda)$, $W^\st(\Lambda)$ have a   transverse intersection along a manifold $\Gamma$, which   is also transverse to the foliation of the unstable manifold
by unstable manifolds of  points.
That is, we will assume it is
transverse to the foliation of the unstable manifold of $\Lambda$ whose leaves are $\{W^{\un}(x)\}_{x\in\Lambda}$.

That is, we will assume that  for each $x\in \Gamma$, we  have
both:
\begin{equation}\label{goodtransversal}
\begin{split}
&T_x\Gamma = T_xW^\st(\Lambda)\cap
T_xW^\un(\Lambda),\\
&T_x W^\un(\Lambda) = T_x \Gamma \oplus  T_xW^\un(x^-),\\
\end{split}
\end{equation}
where $x^-=\Omega^\un(x)$.
The meaning of the second condition will become
apparent later. It is
that the manifold $\Gamma \subset W^\un(\Lambda)$
is transversal to the foliation of $W^\un(\Lambda)$
given by the stable manifolds of
points.  The second part in \eqref{goodtranversal} shows
that locally we can use the Implicit function theorem
to define an  inverse for $\Omega^-$ from a
neighborhood to a neighborhood of $\Gamma$.

We note that the second part of
\eqref{goodtransversal} is implied
by the more symmetric condition:
\begin{equation}\label{goodtransversal2}
T_xM=T_x\Gamma  \oplus
T_xW^\un(x^-)\oplus  T_xW^\st(x^+),
\end{equation}
where $x^+=\Omega^\st(x)$.
Indeed, in the applications, we will verify
\eqref{goodtransversal2}.

Then we can define a map
\[S : U^- \subseteq \Lambda \to U^+ \subseteq
\Lambda  \] associated to $\Gamma $, which we call the
scattering map,  given by \[S(x^-) = x^+,\] provided that
$W^\un(x^-)$ intersects $W^\st(x^+)$ at a (unique) point
$x\in\Gamma$.

The domain $U^-$ of the scattering map $S$ is in general an
open subset of $\Lambda$
where $\Omega^\un$ is invertible from
$\Gamma$ to $U^-$.
Note that, by the implicit function theorem,
if there is a point satisfying
\eqref{goodtransversal}, we can
find a small enough neighborhood $U^-$
where indeed $\Omega^\un$ is invertible.
In this case, we can write the scattering map
as
\[
S = \Omega^\st \circ (\Omega^{\un})^{-1}.
\]

Note that the scattering map depends very much on the
homoclinic intersection $\Gamma$ considered, even if we will omit
it from the notation when it cannot lead to confusion.

If we consider a smooth family of
flows that depend, smoothly on parameters,
 due to the smooth dependence of the stable/unstable
manifolds on parameters and the transversality conditions
\eqref{goodtransversal}
assumed in the definition  of the scattering map, we
obtain the smooth dependence on parameters of the scattering map
of an intersection that depends smoothly on parameters.

Note that, the theory of the scattering map
presented so far applies to any differentiable enough
dynamical system.
When we consider families of Hamiltonian flows,
there are more properties.

\begin{thm}[\cite{DelshamsLS08a}]
\label{scatteringexact}
Assume that the flow $\phi^t$ is Hamiltonian
and that the manifold $\Lambda$ is symplectic
for the restriction. Then, the scattering map is
symplectic; if the flow is exact Hamiltonian,
the symplectic map is exact symplectic.
\end{thm}

One reason why Theorem~\ref{scatteringexact} is useful
is that, as it well known in symplectic
geometry, to determine a family of
symplectic mappings (often called a
symplectic deformation),  it suffices to specify a
family of functions (the Hamiltonian of the deformation).
We note that a family of functions is much easier
than a family of diffeomorphisms. Of course,
it has less components, but, more importantly,
functions transform in a much simpler way
(there are no Jacobians of the change of variable
that come into play).

More precisely:
\begin{prop} \label{deformation}
Let  $s_\eps$ be a differentiable family of diffeomorphisms.
It is exact symplectic if and only if:
\begin{itemize}
\item
$s_0$ is exact symplectic.
\item
Denoting by $\frac{d}{d \eps} s_\eps = \mathcal{S}_\eps \circ s_\eps$,
the vector field $\mathcal{S}_\eps$  (called the generator of
the deformation) satisfies
that there is a function (called the Hamiltonian of the deformation)
such that
\begin{equation}\label{hamiltoneq}
\imath( \mathcal{S}_\eps) \Upsilon = d S_\eps,
\end{equation}
\end{itemize}
where $\Upsilon$ is the symplectic form and $\imath$
is the standard contraction of a vector field and a form.
\end{prop}

The proof of Proposition~\ref{deformation} is
more or less contained in \cite{Moser65}.
See more discussions and the development of a calculus of
deformations in \cite{LlaveMM86, BanyagaLW96}.

To apply the theory of deformations to the scattering map, we
have  the annoying complication that  the scattering maps
are defined on different manifolds. One observation is that,
by the implicit function theorem,
we can find a unique smooth family of mappings

\[
k_\eps: \Lambda_0 \rightarrow
\Lambda_\eps
\]
\begin{equation}
\label{normalization}
\frac{d}{d \eps} k_\eps   \in \text{Span} (E^\st_\eps \oplus E^\un_\eps).
\end{equation}

For the purposes of this paper, which is only perturbation
theory,  the only case that we need the above two
statements for $\eps = 0$. Note that the second
part  Proposition~\ref{symplecticscattering}
only concerns the case $\eps = 0$.

Note that the mappings $k_\eps$ intertwine the dynamics of
the mapping $f_\eps$ with some inner dynamics on $\Lambda_0$.
That is:
\[
f_\eps \circ k_\eps = k_\eps \circ r_\eps
\]
for some smooth family of diffeomorphisms
 $r_\eps: \Lambda_0 \rightarrow \Lambda_0$.

The following result can be found in \cite{DelshamsLS08a}

\begin{prop}\label{symplecticscattering}
Under the assumptions \textbf{H1-H4}:

The mappings $k_\eps$ normalized as in
\eqref{normalization}
 are symplectic in the sense
that
\[
k_\eps \Upsilon|_{\Lambda_0} = \Upsilon|_{\Lambda_\eps}
\]

Then, the reduced scattering map
$\tilde s_\eps = s_\eps \circ k_\eps^{-1}  $ is
symplectic and, for $\eps = 0$, the generator is
the given by
\begin{equation}\label{firstorderscattering}\begin{split}
\int_{-\infty}^\infty \left[ H_1( \phi^t((\Omega^s_\eps)^{-1}(x))) - H_1( \phi^t(x))\right]  \, dt
\end{split}
\end{equation}
\end{prop}

Notice that, because the above expression is a first order calculation, and the integrands converge very
fast, we can substitute for the homoclinic orbit $\Omega^s_\eps$ the limiting value at $\eps = 0$, which is
precisely obtained taking the time $\tau^*$ in the homoclinic intersection.

If one performs the symplectic extension described at the beginning and applies the formula
\eqref{firstorderscattering}, one obtains \eqref{eqn:melnikov_int}.

Notice that, compared with Theorem~\ref{prop:melnikov_jump},
Proposition~\ref{symplecticscattering} also provides some information
on the fast variables.

We refer to \cite{DelshamsLS08a}
for a proof of Proposition~\ref{symplecticscattering},
including the fact that the leading term of the deformation is
the Melnikov integral.

We just mention that the fact that the leading term of the deformation
is given by the Melnikov integral can be made plausible
by the following heuristic argument
(even if it is appealing, we do not know how
to turn it into a proof).
 Recall that the perturbation
theory exists and the first order perturbation should
be a linear functional on the perturbing hamiltonian.
It is reasonable to assume that it depends only on the values
of the Hamiltonian  (and may be its derivatives),
evaluated on the connecting orbit and on the segments of
the asymptotic orbits in the manifold.  To be a generator, it has to
have units of action.
Hence, it has to be
an integral over the connecting trajectory, presumably with weights.
By the invariance with respect to time, we obtain that the weight has
to be constant along the connecting orbit and
the segments  of the unperturbed orbits
in the manifold.  To obtain a convergent integral, we see that we have to
take the weights with one sign in the connecting orbit and with the
opposite sign in the asymptotic orbits. The above heuristic
argument shows that the generator of the family of scattering maps
has to be the Melnikov potential (up to a constant factor).

\section{Some structure on the time dependence of
the perturbations}
\label{sec:nonautonomous}

In many situations appearing in practice, the perturbations
are not arbitrary functions in time, but have some extra
structure such as being quasiperiodic.

In this section, we sketch a geometric formalism in which
the dependence is generating by another flow on another manifold.
This includes as particular cases the periodic and quasi-periodic
perturbations, but also other cases such as chaotic driving by an
Anosov system. It also allows to formulate the fact that the perturbations
may have recurrence in the time dependence. In \cite{GideaL17} it is
shown that, in certain cases, some mild recurrence is enough to
generate Arnol'd diffusion.  Indeed, the formalism described in this
section is the one considered in \cite{GideaL17}, but it also standard
in the theory of time-dependent dynamical systems.

We  consider a \emph{clock manifold} $N$, whose points
which we denote by $\eta$, evolve according to a vector field $\U$,
which generates a flow $\chi^t$.
We will need to assume that the flow $\chi^t$ has a sufficiently small
growth rate so that $\Lambda_0 \times N$ is a
a normally hyperbolic manifold  for the product flow $\phi^t \times \chi^t$.

\begin{equation}\label{eqn:hamiltonian2}
\begin{split}
H_\eps(p,q,I,\vphi,t)=h_0(I)+\sum
_{i=1} ^{n}\pm\left (\frac{1}{2}p_i^2+V_i(q_i)\right)+\eps
h(p, q, I,\vphi, \chi^t(\eta_0); \eps).
\end{split}
\end{equation}
and the Melnikov function etc. will depend on $\eta_0$.

Some important particular cases of
the formalism are periodic and quasi-periodic flows.
In the periodic case, $N = \torus$ and the flow $\chi^t$ is just
$\U( \eta) = 1$. In the case of quasi-periodic perturbations,
we have that $N = \torus^m$ and the flow is just rotation
of constant speed.

We note that we do not need that the manifold $N$ is compact,
but we will need that the clock flow is Lipschitz with a
small constant. For example, if we ignore the recurrence properties
and take $N = \real $ and $\U(\eta)= 1$.

In this formalism, the Melnikov vector and the Melnikov potential
become respectively,
\begin{equation} \label{eqn:melnikov_vect2}
\begin{split}
\Mv_i(\tau, I, \vphi, \eta_0) =
\displaystyle-\int
_{-\infty}^{\infty}&\left[
(\X^1 P_i)(p^0(\tau+\sigma\bar 1),q^0(\tau +\sigma\bar 1),I,\vphi+\sigma\omega(I), \chi^\sigma(\eta_0);0)\right .\\
&\left .-(\X^1 P_i)(0,0,I,\vphi+\sigma\omega(I),\chi^\sigma(\eta_0);0)\right
]d\sigma.
\end{split}
\end{equation}

\begin{equation}\label{eqn:melnikov_int2}
\begin{split}
\tilde{\mathcal{M}}(\tau,I,\vphi,\eta_0)=
\displaystyle-\int
_{-\infty}^{\infty}&\left
[h(p^0(\tau+\sigma\bar 1),q^0(\tau +\sigma\bar 1),I,\vphi+\sigma\omega(I),\chi^\sigma(\eta_0);0)\right .\\
&\left .-h(0,0,I,\vphi+\sigma\omega(I),\chi^\sigma(\eta_0); 0)\right
]d\sigma.
\end{split}
\end{equation}

The proof of the formula \eqref{eqn:melnikov_vect} follows
by applying the argument in the text to the extended system $\phi^t \times \chi^t$
taking into account that the $\Lambda_0 \times N$ is a normally
hyperbolic manifold and that the perturbations have smooth
manifolds  (here we use the fact that the growth rates of
the flow $\chi^t$ are small enough).

\bibliographystyle{alpha}
\bibliography{diffusion}

\end{document}